\documentclass[11pt,a4paper, envcountsame]{amsart}
\usepackage[usenames,dvipsnames]{color}

 \usepackage[colorlinks,citecolor=blue,urlcolor=blue, linkcolor=blue, backref]{hyperref}

 \usepackage{amsmath, amssymb, xspace}
 \usepackage{graphicx}

 \usepackage[all]{xy}

\usepackage{tikz}
\usetikzlibrary{arrows,snakes,positioning,backgrounds,shadows}

\usepackage{undertilde}

\newtheorem{theorem}{Theorem}[section]

\newtheorem{thm}[theorem]{Theorem}

\newtheorem{proposition}[theorem]{Proposition}

\newtheorem{claim}[theorem]{Claim}

\newtheorem{lemma}[theorem]{Lemma}

\newtheorem{cor}[theorem]{Corollary}

\newtheorem{question}[theorem]{Question}

\theoremstyle{definition}

\newtheorem{remark}[theorem]{Remark}

\newtheorem{definition}[theorem]{Definition}

\newcommand{\QI}{\Pi^1_1}
\newcommand{\VI}{\Sigma^1_1}
\newcommand{\OMC}{\omega_1^{CK}}

\newcommand{\DII}{\Delta^0_2}
\newcommand{\NN}{{\mathbb{N}}}
\newcommand{\RR}{{\mathbb{R}}}
\newcommand{\R}{{\mathbb{R}}}

\newcommand{\ZZ}{{\mathbb{Z}}}
\newcommand{\sub}{\subseteq}
\newcommand{\sN}[1]{_{#1\in \NN}}
\newcommand{\uhr}[1]{\! \upharpoonright_{#1}}
\newcommand{\ML}{Martin-L{\"o}f}
\newcommand{\SI}[1]{\Sigma^0_{#1}}
\newcommand{\PI}[1]{\Pi^0_{#1}}
\newcommand{\PPI}{\PI{1}}

\newcommand{\bi}{\begin{itemize}}
\newcommand{\ei}{\end{itemize}}
\newcommand{\bc}{\begin{center}}
\newcommand{\ec}{\end{center}}

\newcommand{\ES}{\emptyset}
\newcommand{\estring}{\la \ra}
\newcommand{\ria}{\rightarrow}
\newcommand{\tp}[1]{2^{#1}}
\newcommand{\ex}{\exists}
\newcommand{\fa}{\forall}

\newcommand{\la}{\langle}
\newcommand{\ra}{\rangle}
\newcommand{\Kuc}{Ku{\v c}era}
\newcommand{\seqcantor}{2^{ \NN}}

\newcommand{\cantor}{\seqcantor}
\newcommand{\strcantor}{2^{ < \omega}}

\newcommand{\Opcl}[1]{[#1]^\prec}

\newcommand{\MLR}{\mbox{\rm \textsf{MLR}}}
\newcommand{\WTR}{\mbox{\rm \textsf{W2R}}}
\newcommand{\DR}{\mbox{\rm \textsf{DenseR}}}
\newcommand{\Om}{\Omega}
\newcommand{\n}{\noindent}

\newcommand{\vsps}{\vspace{3pt}}

\newcommand{\vsp}{\vspace{6pt}}
\newcommand{\leb}{\mathbf{\lambda}}

\newcommand{\sss}{\sigma}
\newcommand{\aaa}{\alpha}

\newcommand{\lland}{\, \land \, }

\newcommand \seq[1]{{\left\langle{#1}\right\rangle}}

\newcommand\+[1]{\mathcal{#1}}
\newcommand{\sC}{\+ C}

\newcommand{\wt}{\widetilde}
\newcommand{\ol}{\overline}
\newcommand{\ul}{\underline}
\newcommand{\ape}{\hat{\ }}

\newcommand{\LR}{\Leftrightarrow}
\newcommand{\RA}{\Rightarrow}
\newcommand{\LA}{\Leftarrow}

\newcommand{\rapf}{\n $\RA:$\ }
\newcommand{\lapf}{\n $\LA:$\ }

\newcommand{\sssl}{\ensuremath{|\sigma|}}

\newcommand{\weight}{\mbox{\rm \textsf{wt}}}

\newcommand{\real}{\mathbb{R}}

\newcommand{\upr}{\upharpoonright}

\usepackage[utf8]{inputenc}
% ...
\inputencoding{latin1}
% ...
\inputencoding{utf8}
% ...

\begin{document}

\thanks{Miyabe was supported by JSPS. Nies was supported by the Marsden fund of New Zealand. Zhang was supported by a University of Auckland summer scholarship.}
 
\title[Using  theorems from analysis to study randomness]{Using almost-everywhere theorems \\  from analysis to study randomness}

\author{Kenshi Miyabe, Andr\'e Nies, and Jing Zhang}
\begin{abstract}  We study algorithmic randomness notions via effective versions of almost-everywhere theorems from analysis and ergodic theory. The effectivization is in terms of  objects described by a computably enumerable set, such as lower semicomputable functions. The corresponding randomness notions are slightly stronger than \ML\ (ML) randomness. 

We establish several equivalences. Given a  ML-random real $z$,  the additional randomness strengths  needed for  the following  are equivalent.

\n  (1) all effectively closed classes containing $z$ have density $1$ at $z$. 

\n   (2)  all nondecreasing functions with uniformly  left-c.e.\ increments are differentiable at $z$. 

\n  (3) $z$ is a Lebesgue point of each lower semicomputable integrable function. 

We also consider  convergence of left-c.e.\ martingales,  and convergence in the sense of Birkhoff's pointwise ergodic theorem. Lastly,  we study randomness notions related to  density of $\Pi^0_n$  and $\Sigma^1_1$ classes at a real.
\end{abstract}
\maketitle

\tableofcontents
\section{Introduction} Several theorems  in analysis  and ergodic theory   express that all functions in a certain class are well-behaved at almost every point. For instance,   Lebesgue published the following  theorem in 1904. It is often  covered in textbooks on analysis, e.g.\  \cite[Ch.\ 20]{Carothers:00}.
 
\begin{thm}[\cite{Lebesgue:1909}]  \label{thm:Leb} Let  $f \colon \, [0,1] \to \RR$ be a nondecreasing function. Then $f$  is differentiable almost-everywhere. \end{thm}  Another example of such a result is Birkhoff's ergodic theorem;  see e.g.\  \cite[Thm.\ 2.3]{Krengel:85} for a textbook reference.

\begin{theorem}[\cite{Birkhoff:39}]  \label{thm:Birkh} Let  $T$ be   a  measure preserving operator  on a probability space $X$.  Let $f$ be    an  integrable function on $X$. Then    for almost every point $z \in X$,   the average of $f(z), f(T(z)), \ldots, f(T^{n-1}(z))$    converges as $n \to \infty$.  If the operator is ergodic then this limit is the integral of $f$. \end{theorem}

The  theorems  involve   a null set of exceptions which usually  depends on  given objects, such as   $T$ and $f$ in Theorem~\ref{thm:Birkh}. By an effective version of such a theorem, we mean the following. If the given   objects are  algorithmic  in some sense, then   the resulting   null set is  also algorithmic.  (A~slightly stronger effective version of such a  theorem   would also ask that the null set be obtained uniformly from a presentation of the given objects, without the assumption that it  is algorithmic; this is  usually  the case for the examples we   consider.)

  Brattka, Miller and Nies \cite{Brattka.Miller.ea:16}, in their Thm. 4.1 combined with Remark 4.7,  show the following effective version of   Lebesgue's theorem.  The given object is a computable function. 
\begin{theorem}[\cite{Brattka.Miller.ea:16}]\label{thm: comp dyadic diff}
	Suppose a nondecreasing  function $f \colon \, [0,1] \to \RR$ is computable.  There  exists  a computable martingale that succeeds on the binary presentation of each real $z$ such that $f'(z)$ fails to exist.   
\end{theorem} 
We explain the terms used in this theorem.

\vsp

\n  (a)  The  computability of a function is taken in the usual sense of computable analysis \cite{Weihrauch:00}. As shown in the last section of the  longer arXiv version of~\cite{Brattka.Miller.ea:16}, the weaker hypothesis is sufficient that $f(q)$ be a computable real uniformly in a rational~$q$. 

\n  (b) In randomness theory, a  \emph{martingale} is a function $M \colon \strcantor \to \RR^+_0$ such that  $2 M(\sss)= M(\sss0) + M(\sss 1)$. A martingale $M$ succeeds on a bit sequence~$Z$ if the value of $M$ on initial segments of $Z$ is unbounded.  The success set is a null set which is effective in case $M$ is computable.  

\vsp

A  real on which no computable martingale succeeds is called computably random, a notion introduced by Schnorr~\cite{Schnorr:75}; for a recent  reference see  \ \cite[Ch.\ 7]{Nies:book} or \cite{Downey.Hirschfeldt:book}.  The theorem above shows that $f'(z)$ exists for each computably random real $z$ and each nondecreasing computable function $f$. Brattka et al.\ also show that conversely, if a real $z$ is not computably random, then some computable monotonic function $f$  fails to be differentiable at~$z$.   
In this way,  this effective form of Lebesgue's Theorem~\ref{thm:Leb}  is matched to computable randomness. 
This is  an instance of a more general principle: effective versions of ``almost-everywhere''   theorems often  correspond to well-studied algorithmic randomness notions.

	Pathak, Rojas and Simpson~\cite[Theorem 3.15]{Pathak.Rojas.ea:12} matched a particular effective form of the Lebesgue differentiation theorem to Schnorr randomness (the  direction where a  function is turned into a  test was independently proven in   \cite[Thm.\ 5.1]{Freer.Kjos.ea:14}). We will discuss this in more detail in 	Subsection~\ref{ss:Leb-L1-comp}.

V'yugin~\cite{Vyugin:98}, G\'acs et al.\ \cite{Gacs.Hoyrup:11}, Bienvenu et al.\ \cite{Bienvenu.Day.etal:12}, Franklin et al.\ \cite{Franklin.Greenberg.etal:12}, and 	Franklin and Towsner \cite{Franklin.Towsner:14} all studied  effective versions of   Birkhoff's theorem. For instance, in the notation above, if  an ergodic operator $T $ is computable, and the integrable function $f$ is lower semicomputable as defined below, then the corresponding notion is \ML\ randomness by~\cite{Bienvenu.Day.etal:12,Franklin.Greenberg.etal:12}. 
	
	Matching such theorems to  algorithmic randomness notions has been useful in two ways: 
\bi \item[(a)]  to determine  the   strength of the theorem, and
 \item[(b)]  to   understand   the randomness notion. \ei 
  For an example of (a), Demuth~\cite{Demuth:75} (see \cite{Brattka.Miller.ea:16} for a proof in modern language)  showed that   Jordan's  extension of Lebesgue's result to functions of bounded variation corresponds to \ML\ randomness. This notion   is stronger than computable randomness; so in a sense this extension is   harder to obtain.  For an example of (b), Brattka et al.\ \cite{Brattka.Miller.ea:16} used their results to show that computable randomness of a real does not depend on the choice of base in its digit expansion, even though martingales (which can also be defined with respect to bases other than 2) bet on such an expansion.

The main purpose of this paper is to examine effective versions of almost-everywhere theorems that do not correspond to  known randomness notions.    This apparently occurred for the first time when Bienvenu et al.\  showed in \cite[Cor.\ 5.10]{Bienvenu.Hoelzl.ea:12a}     that the randomness notion corresponding to the Denjoy-Young-Saks theorem implies computable randomness, but is incomparable with \ML\ randomness.  

 We   base our study on Lebesgue's theorems mentioned earlier, and on  the following two results. The first,   Lebesgue's density theorem~\cite{Lebesgue:1909},  asserts that for almost every point $z$   in a measurable class $\+C \sub [0,1]$, the class is ``thick'' around $z$ in the sense that the relative measure of $\+ C$ converges to $1$ as one ``zooms in'' on $z$.     The second,   Doob's   martingale convergence theorem  \cite{Durrett:96},   says that a martingale converges on almost every point.  

The main given  object  will   only be effective in the  weak  sense of computable enumerations. We   consider the Lebesgue density theorem for  effectively closed sets of reals (the complement is an open set that can be computably enumerated as a union of rational open intervals). We consider  Doob's convergence  theorem  for martingales that uniformly assign  left-c.e.\ reals to  strings.  
  
  A group or researchers working at the University of Wisconsin at  Madison,  consisting of Andrews, Cai, Diamondstone, Lempp, and  Miller,    showed in 2012
   that for a real $z$ the following two conditions are equivalent, thereby connecting the two theorems.

\vsp 

\n (1) $z$ is \ML\ random and every effectively closed  class containing $z$ has density 1 at $z$

\vsp

\n (2)  every left-c.e.\ martingale converges along the binary expansion of $z$. 

\vsp 

 In this paper we provide  two further  conditions on a real $z$ that are  equivalent to the ones above. They are  are also linked to  well-known classical results of the   ``almost-everywhere'' type where the  main given  object is  in some sense computably enumerable.  The conditions are:

\vsp 

\n (3) every interval-c.e.\ function $f$ is differentiable at $z$

\vsp

\n (4)   $z$ is \ML\ random  and a Lebesgue point of each  integrable  lower semicomputable function $g \colon [0,1] \to  {\RR} \cup \{\infty\}$.

\vsp
 
By default, functions will have domain $[0,1]$. In (3), the relevant classical result is  Lebesgue's  theorem on monotonic functions discussed above.   To say that a monotonic function~$f$ is interval-c.e.\ means that $f(0)=0 $ and $f(q) -f(p)$ is left-c.e.\ uniformly in rationals $p<q$. In (4), the classical result is  Lebesgue's differentiation theorem, which extends the density theorem.   A function  $g$ is lower semicomputable if $\{x\colon g(x) > q \}$ is $\SI 1$ uniformly in a rational $q$.  

The new randomness notion identifying the strength of each of the conditions (1)--(4)   will be called \emph{density randomness}.

The analytic notion of density   has already been  very useful for resolving  open problems on the complexity of sets of numbers,  asked for instance in~\cite{Miller.Nies:06}. It was applied in~\cite{Day.Miller:14} to show that $K$-triviality coincides with ML-noncuppability. It was further used to solve the so-called covering problem that every $K$-trivial is Turing below an incomplete ML-random oracle, and in fact below a single such oracle that also is~$\DII$. See the survey \cite{Bienvenu.Day.ea:14}   for more detail and references.

Sections 2-6 of the paper are based   on  the  almost-everywhere theorems that serve as an analytic background for our algorithmic investigations: Lebesgue density theorem, Doob martingale convergence, differentiability of monotonic functions~\cite{Lebesgue:1909},   Lebesgue differentiation theorem~\cite{Lebesgue:1910}, and Birk\-hoff's theorem~\cite{Birkhoff:39}.    
In a final section we will   study density  for    classes that have descriptional  complexity  higher  than~$\PPI$.   

This work is a mix of survey and research paper. Section~\ref{s:LebD} introduces the notion of density of a class at a point in detail, and contains basic  results on effective aspects of density, some of them  new. Section~\ref{s:MG convergence} contains a proof of the  unpublished 2012 result of the Madison group (with permission). Section~\ref{s:STACS14} elaborates on a conference paper of Nies \cite{Nies:14}.  The remainder  of the paper consists of new results.
\section{Lebesgue density theorem} \label{s:LebD}
\newcommand{\cdf}{{\sf cdf}}
This section   presents background material and some initial  results. We discuss the theorem that leads to the definition of two central  notions for this paper, density-one points and density randomness.  We also look at these notions  in the setting of Cantor space. M.\ Khan and J.\  S.\ Miller (see \cite{Khan:15}) have shown that among the ML-random reals, this choice  of a  setting does not make a difference.  We show that lowness for density randomness is the same as lowness for ML-randomness, or equivalently, $K$-triviality.

\subsection{Density in the setting of reals} 
%
%
The definitions  below   follow \cite{Bienvenu.Hoelzl.ea:12a}. Let~$\lambda$ denote Lebesgue measure. 
\begin{definition}
We define the lower Lebesgue density of a set $\sC \subseteq \R$ at a point~$z$ to be the quantity $$\ul \varrho(\sC | z):=\liminf_{\gamma,\delta \rightarrow 0^+} \frac{\lambda([z-\gamma,z+\delta] \cap \sC)}{\gamma + \delta}.$$
\end{definition} 
%
 Note that  $0 \le \ul \varrho(\sC | z) \le 1$.
  \begin{theorem}[Lebesgue \cite{Lebesgue:1909}] 
Let $\sC \subseteq \R$ be a measurable set. Then  $\ul \varrho(\sC | z)=1$ for almost every $z\in \sC$.
\end{theorem}

When $\sC$ is open, then the lower Lebesgue density is clearly $1$.
Thus, the simplest non-trivial case is when $\sC$ is closed. We use this case to motivate our central definition.
\begin{definition}  \label{density random}
  We say that a real  $z\in[0,1]$  is a \emph{density-one point}   if $\ul \varrho(\sC | z)=1$ for every effectively closed class $\sC$ containing $z$.   We say that $z$ is \emph{density random} if $z$ is a  density-one point and  \ML\ random.
\end{definition}
As noted e.g.\ in \cite{Bienvenu.Hoelzl.ea:12a}, being a density-one point by itself is not a   reasonable randomness notion:   for instance, every  1-generic real is   a density-one point, but fails the law of large numbers.

By the Lebesgue density theorem and the fact that there are only countably many effectively closed classes, almost every real $z$ is density random.  
Recall that a real is  weakly-2-random  if it does not lie in any $\PI 2$ null class. In fact, any such real is density random: for any effectively closed $\+ C$ and  rational $q< 1$, the null class $\{ z\in \+ C \colon \, \ul \varrho (\+ C \mid z)  \leq  q \}$ is $\PI 2$.

  We say that $z$ is a \emph{positive density point} if $\ul \varrho(\sC | z)>0$ for every effectively closed class $\sC$ containing $z$. The difference between positive and full density is typical for our algorithmic setting. In classical analysis, null sets are usually negligible, so everything is settled by Lebesgue's theorem. In effective analysis, a result of      Day and Miller~\cite{Day.Miller:15} separates the two cases:  for a ML-random   real $z$,  to be  a  full density-one point is a stronger  randomness condition  than to be a  positive density point.

  Bienvenu et al.\ \cite{Bienvenu.Hoelzl.ea:12a} have shown  that a ML-random  real $z$ is a positive density point   if and only if $z$   is Turing incomplete. In contrast, for density-one points, no characterisation in terms of computational complexity  among the ML-random reals   is   known at present.

\subsection{Density in the setting of Cantor space}
We let $\cantor$ denote the usual product probability space of infinite bit sequences. For $ Z \in \cantor $ we let $Z \uhr n$ (or $Z\upr n$ in subscripts) denote the first $n$ bits of~$Z$. Variables $\sss, \tau,  \eta$ range over strings in $\strcantor$.  We denote by  $\sss \preceq \tau$ that $\sss$ is an initial segment of $\tau$;   $\sss \prec \tau $ denotes  that $\sss$ is a  proper initial segment of $\tau$; $\sss \prec Z $ that $\sss$ is an initial segment of the infinite bit sequence~$Z$.  

 For  each  $\sss $ we let $[\sss]$ denote the clopen set of extensions of $\sss$.  For $\sC \sub \cantor$ we let $\leb_\sss(\sC) = \tp \sssl  \leb(\sC \cap [\sss])$ denote  the local measure of $\sC$ inside~$[\sss]$.

Consider a  measurable set $\+ C\subseteq \cantor$ and $Z\in \cantor$. The \emph{lower density} of $Z\in \cantor$ in $\+ C$ is defined to be \[\ul \varrho_2(\+ C|Z)=\liminf_{n \to \infty} \leb_{Z \upr n} (\+ C) \]  
We say that a real $z \in [0,1]$ is a \emph{dyadic density-one} point if its dyadic expansion is a density one point in Cantor space. We will use  the following result.
\begin{theorem}[Khan and  Miller  \cite{Khan:15}]\label{th:ML-dyadic-full}
Let $z$ be a ML-random dyadic density-one point.
Then $z$ is a full density-one point.
\end{theorem}
Thus,  by the usual  identification of  irrational real numbers in $[0,1]$ with elements in Cantor space,  we can   equivalently  define density randomness for a real as in Definition~\ref{density random}, or    for the corresponding bit sequence in Cantor space using lower dyadic density.

\subsection{Lowness for Density randomness}
We say that a Turing oracle $A$ is \textit{low for density randomness} if whenever $Z \in \cantor $ is density random, $Z$ is already density random relative to $A$.
Here, $z$ is density random relative to $A$ if $z$ is ML-random relative to $A$, and 
  $\ul \varrho(\sC | z)=1$ for every $A$-effectively closed class $\sC$ containing $z$. 
We will show that this is equivalent to lowness for ML-randomness.
  
 By  $\WTR$  we denote  the class of   weakly-2-random sets, i.e.\  sets that do not lie in any $\Pi_2^0$-null class of sets. Low(\WTR, \MLR) denotes the class of oracles $A$ such that $\WTR \sub \MLR^A$. Downey, Nies, Weber and Yu \cite{Downey.Nies.ea:06} have shown that Low(\WTR, \MLR)=Low(\MLR)

\begin{lemma}[Day and Miller \cite{Day.Miller:14}] \label{TrivialRandom}
Suppose $Z$ is \ML \ random, $A$ is low for ML-randomness,  and $\mathcal{P}$ is a $\Pi_1^{0,A}$ class containing $Z$. Then there exists a  $\Pi_1^0$ class $\mathcal{Q}\subseteq \mathcal{P}$ such that $A\in \mathcal{Q}$.
\end{lemma}

\begin{theorem}
   $A\in \cantor$ is low for ML-randomness $\LR$  
   
   \hfill $A$ is low for density randomness.
\end{theorem}
\begin{proof}
\lapf  Let \DR \ denote the class of density random sets. Since    $\WTR \subseteq \DR \subseteq \MLR$ and by the result in \cite{Downey.Nies.ea:06}, we have  \bc Low(\DR) $\subseteq$ Low(\WTR, \MLR)=Low(\MLR). \ec

\rapf  Suppose that $A$ is not low for density randomness, i.e., there exists a set $Z$  that is density random but not density random relative to $A$. If  $Z$ is not even  \ML \ random relative to $A$, then $A$ is not low for ML-randomness.
Otherwise, $Z$ is \ML \ random relative to $A$ but not density random relative to $A$. Hence   there exists a $\Pi_1^{0,A}$ class $\mathcal{P}$ containing $Z$ such that $\ul \varrho_2(\mathcal{P}|Z)<1$. By Lemma \ref{TrivialRandom}, there is  a  $\Pi_1^0$ class $\mathcal{Q}\subseteq \mathcal{P}$ such that $Z\in \mathcal{Q}$.   Then $\ul \varrho_2(\mathcal{Q}|Z)\leq  \ul \varrho_2(\mathcal{P}|Z)<1$,  so $Z$ is not density random, contradiction.
\end{proof}

\subsection{Upper density} The {upper}   density of   $\sC \subseteq \cantor$ at~$Z$ is: 
\[\ol \varrho_2(\sC|Z)=\limsup_{ n \to \infty }\leb_{Z \upr n} (\+ C)\]  
  Bienvenu et al.\ \cite[Prop.\ 5.4]{Bienvenu.Greenberg.ea:nd} have shown that for any effectively closed set $\+ P$ and   ML-random $Z \in \+ P$, we have $\ol \varrho_2(\+ P\mid Z) =1$. Actually
 ML-randomness of $Z$ was  too strong an assumption. The weaker notion of  partial computable randomness, defined in terms of partial computable martingales, already suffices. See \cite[Ch.\ 7]{Nies:book} for background on this notion.  

\begin{proposition} \label{prop:PC random upper density}  Let $\+ P \sub \cantor$ be effectively closed. Let  $Z \in \+ P$. 

\bi \item[(i)] If $Z$ is  partial computably random, then $\ol \varrho_2(\+ P \mid Z) =1$. 

\item[(ii)] Suppose that, in addition,  $\leb \+ P$ is computable. If $Z$ is Kurtz random, then $\ol \varrho_2(\+ P \mid Z) =1$.  \ei \end{proposition}

\begin{proof} Suppose  that  there is a  rational  $q < 1$ and an  $n^*\in \NN$ such  that $\leb_\eta(\+ P) < q$ for each $\eta \prec Z$ with $|\eta| \ge n^*$. 

\n (i).   We   define a partial computable martingale $M$ that succeeds on $Z$. 
Let $M(\eta)=1 $ for all strings $\eta$ with $|\eta| \le n^*$. Now suppose that   $M(\eta)$ has been defined, but $M$ is as yet undefined on any extensions of $\eta$. 
Search for $t= t_\eta > |\eta|$ such that $  p: =   |F|\tp{-(t -|\eta|)} \le  q$, where  \bc $F = \{\tau \succ \eta \colon  \, |\tau| = t   \lland [\tau] \cap \+ P _t \neq \ES\}$. \ec 
If $t_\eta$ and $F$ are  found, bet all the capital existing at $\eta$ along the strings in~$F$. That is, for $\tau \succeq \eta$, $|\tau| \le t$, let 
\bc$M(\tau) = M(\eta) \cdot| \{ \sss \in F \colon \, \sss \succeq \tau\}|/p$. \ec Then $M(\sss) = M(\eta)/p \ge M(\eta)/q$ for each $\sss \in F$. Now continue the procedure with all such strings $\sss \succ \eta$  of length $t$.

For each $\eta \prec Z$ of length at least $n^*$, we have $\leb_\eta(\+ P) < q$, so a $t_\eta$ as above  will be found. Since  $ Z \in \+ P$, $M$ never decreases along $Z$. Then,  since  $q < 1$, $M$ succeeds on $Z$.

\n (ii).   Under the extra hypothesis on $\+ P$, we can make $M$ total,  and also  bound from below  its growth   at an infinite computable set of positions along $Z$. This will  show that $Z$ is not Kurtz random (see Downey and Hirschfeldt \cite[Theorem 7.2.13]{Downey.Hirschfeldt:book}).  

Note that $\leb_\eta(\+ P) $ is a  computable real uniformly in $\eta$. Pick  rationals  $q' < q < 1$ and an  $n^*\in \NN$ such  that $\leb_\eta(\+ P) < q'$ for each $\eta \prec Z$ with $|\eta| \ge n^*$. In the same situation as above, search for $t_\eta > |\eta|$ such that we see $\leb_\eta(\+ P) > q'$ at stage $t_\eta$, or $F$ is found. One of the cases must occur. If the former case is seen first, we let $M(\tau) = M(\eta)$ for all $\tau \succ \eta$, $\tau \le t_\eta$. Otherwise,  we proceed as above.   

For the lower bound on the growth, define a computable function by \bc  $g(n) = \max \{t_\eta \colon \, n^* \le |\eta| \le  n \}$, \ec
for $n \ge n^*$,  and $g(n) =0$ otherwise. Let $r(k)= g^{(2k)}(n^*)$. Then 

\n $M(Z \upr {r(k)}) \ge q^{-k}$ for each $k$.
\end{proof}

It is not known at present whether the partiality of $M$  in (i)  is necessary.
\begin{question} Is there  a  $\PPI$ class $\+ P$  and a computably random $Z \in \+ P$  such  that $\ol \varrho_2(\+ P \mid Z) < 1$ ? \end{question}
In Subsection~\ref{ss:Leb-L1-comp} we  will  continue the  study of   $\PPI$ classes  of computable measure. We  show that such a class has density one at every  Schnorr random member.

\section{Martingale convergence theorem}  \label{s:MG convergence}

For background on martingales in probability theory, see for instance Durrett \cite[Ch.\ 4]{Durrett:96}. The martingale convergence theorem goes back to work of Doob.  Recall that  for a random variable $Y$ one defines $Y^+ = \max (Y,0)$.
\begin{theorem}  \label{thm:mg_convergence}  Let $\seq{X_n}\sN n$ be a martingale with $\sup_n E X_n^+ < \infty$. Then $X(w):= \lim_n X_n(w)$ exists almost surely, and $E|X| < \infty$.  \end{theorem}

The standard proof (see e.g.\ \cite[Ch.\ 4, (2.10)]{Durrett:96}) uses  Doob's upcrossing inequality.
In randomness theory, researchers have so far only used  the  very restricted form  of the  powerful notion of a  martingale defined in the introduction: The probability space is Cantor space with the usual product measure. The filtration $\seq {\+ F_n}\sN n$ is defined  by letting  $\+ F_n$ be the set of events that   only depend on the first $n$ bits.    If $\seq {X_n}$ is adapted to   $\seq {\+ F_n}\sN n$, then $X_n$ has constant value on each $[\sss]$ for $\sssl =n$. Let $M(\sss) $ be this value. The martingale condition $E(X_{n+1} \mid \+ F_n) = X_n$ now turns into     $ \fa \sss \,  M(\sss0) + M(\sss 1) = 2M(\sss)$.  One also requires that the values be non-negative (so that one can reasonably define that a martingale  succeeds  along a bit sequence).  

Note that $EX_0 = M(\estring) < \infty$. Thus,    Theorem~\ref{thm:mg_convergence} turns into the following. 

\begin{thm} \label{thm:mg_convergence_restr} Let $M \colon \strcantor \to   {\mathbb R}^+_0$ be a  martingale in the restricted sense above. Then for almost every $Z \in \cantor$,    $X(Z) = \lim_n M(Z\uhr n) $ exists and is finite. Furthermore, $EX < \infty$.\end{thm} 

If $\lim_n M(Z\uhr n)$ exists and is finite, we say that $M$ converges \emph{along} $Z$. 

We can now analyze the theorem in the effective setting, according to the main plan of the paper. Firstly we discuss the  effective form of Theorem~\ref{thm:mg_convergence_restr} in terms of computable martingales. It is not hard to show that a  computable martingale converges along any computably random bit sequence $Z$ (see~\cite[Theorem 7.1.3]{Downey.Hirschfeldt:book}). In other words, boundedness of all computable martingales along a bit sequence $Z$  already implies their convergence.  For the converse see the proof of \cite[Thm.\ 4.2]{Freer.Kjos.ea:14}, where success of a computable martingale is turned into oscillation of another. Thus, this  effective form of Theorem~\ref{thm:mg_convergence_restr} is matched to computable randomness.

Next, we weaken the effectiveness  to a notion based on computable enumerability. A martingale $L \colon \strcantor \to \RR^+_0$ is called \emph{left-c.e.} if $L(\sss)$ is a left-c.e.\ real uniformly in $\sss$.  Note that $Z$ is \ML-random iff every such martingale is bounded along $Z$ (see e.g.\ \cite[Prop.\ 7.2.6]{Nies:book}). Unlike the case of computable martingales, convergence requires a stronger form of algorithmic  randomness than boundedness. For instance, let $\+ U = [0, \Om)$ where $\Om$ is a left-c.e.\ \ML-ranom real, and let  $L(\sss) = \leb_\sss(\+ U)$; then the left-c.e.\ martingale $L$ is bounded by 1 but  diverges on $\Omega$ because $\Om$ is Borel normal.

The following theorem  matches left-c.e.\ martingale convergence to density randomness. It is due to unpublished 2012 work of the  ``Madison Group'' consisting of Andrews, Cai, Diamondstone, Lempp and Joseph S.\ Miller.  Recall  that by  Theorem~\ref{th:ML-dyadic-full}, a ML-random  $z$ is a  full density-one point iff $z$ is a dyadic density-one point.

\begin{theorem}[Madison group] \label{thm:Madison} The following are equivalent for a ML-random real $z \in [0,1]$ with binary expansion $0.Z$.
\bi 
\item[(i)] $z$ is a dyadic density-one point.   
\item[(ii)] Every left-c.e.\ martingale converges along $Z$. \ei 
\end{theorem}
 The writeup of the proof below, due to Nies, is based on discussions with Miller, and  Miller's  slides for his talks at a Semester dedicated to computability, complexity  and randomness at Buenos Aires  in 2013 \cite{Miller:13}. Nies supplied the technical details of  the verifications.   
 
\begin{proof}

\n The easier implication  (ii) $\to$ (i) was proved in   \cite[Corollary 5.5]{Bienvenu.Greenberg.ea:nd}.   Simply note that for a $\PPI$ class $\+ P$, the function $M(\sss) = 1 - \lambda_\sss (\+ P)$ is a left-c.e.\ martingale. Convergence of $M$ along $Z$ means that  $\rho(\+ P \mid Z)$ exists. Prop.\ \ref{prop:PC random upper density} implies that the upper density $\ol \rho(\+ P \mid Z)$ equals~$1$. Therefore $\rho(\+ P \mid Z)=1$.

\vsps

\n (i) $\to$ (ii).  We can   work within Cantor space because the  dyadic density  of a class $\+ P \sub [0,1]$ at $z$ is the same as the density of $\+ P$ at $Z$ when $\+ P$ is viewed as a subclass of   Cantor space.  We use the  technical  concept of a ``Madison test''. Such a test is  intended to capture   the oscillation of a  left-c.e.\ martingale along a bit sequence.   We will now  introduce and motivate this     concept. We define the weight  of a set   $X \sub \cantor$  by  
  $${\weight} ( X) = \sum_{\sss \in X} \tp{-\sssl}.$$ 
  Let $\sss^\prec = \{ \tau \in \strcantor \colon \, \sss \prec \tau \}$ denote the set of proper extensions of a string~$\sss$.
\begin{definition}  A \emph{Madison test} is a computable sequence $\seq {U_s}\sN s$ of computable subsets of $\strcantor$  such that $U_0 = \ES$,      there is  a  constant $c$ such that for each  stage $s$ we have $\weight (U_s)  \le c$, and for  all strings $\sss, \tau$, 
\bi \item[(a)]  $\tau \in U_s - U_{s+1} \to \ex \sss \prec \tau \, [ \sss \in U_{s+1} - U_s]$
\item [(b)] $\weight (\sss^\prec \cap U_s) > \tp{-\sssl} \to \sss \in U_s$. 
\ei
Note that by (a),  $U(\sss) := \lim_s U_s(\sss)$ exists for each $\sss$; in fact, $U_s(\sss)$ changes at most $\tp{\sssl}$ times.

We say that $Z$ \emph{fails} the test  $\seq {U_s}\sN s$ if $Z \uhr n \in U$ for infinitely many $n$; otherwise $Z$ \emph{passes} $\seq {U_s}\sN s$.
\end{definition}
 
We show    that $\weight(U_s)\leq \weight(U_{s+1})$, so that  $\weight(U) = \sup_s \weight(U_s) < \infty $  is a left-c.e.\ real.   Suppose  that $\sss$ is  minimal under the prefix relation such that  $\sigma\in U_{s+1}-U_s$.  By (b) and since $\sigma\not \in U_{s}$, we have  $\weight(\sigma^\prec \cap U_s)\leq 2^{-|\sigma|}$. So enumerating $\sss$ adds $\tp{-\sssl}$ to the weight, while the weight of all the strings  above $\sigma$ that are  removed from $U_s$  is at most $2^{-|\sigma|}$.

The implication (i)$\to$(ii) is proved in two steps.
 
\vsps 

\n {\it Step 1.}   Lemma \ref{lem: density to Madison} shows that if $Z \in \cantor$ is a ML-random dyadic density-one point, then $Z$   passes all Madison tests. 

\vsps 

{\n Step 2.}  Lemma~\ref{Madison to  MG convergence} shows  that if $Z$ passes all Madison tests, then every left-c.e.\ martingale converges along $Z$.

\vsps

\begin{lemma} \label{lem: density to Madison} Let $Z$ be a ML-random dyadic density-one point. Then $Z$ passes each Madison test. \end{lemma}

\begin{proof} Suppose that  a ML-random bit sequence $Z$ fails a Madison test $\seq {U_s}\sN s$. We will build a ML-test $\seq {\+ S^k} \sN k$ such that $\fa \sss \in U \, [ \leb_\sss(\+ S^k) \ge \tp{-k}]$, and  therefore $$\ul \varrho(\cantor - \+ S^k \mid Z) \le 1- \tp{-k}.$$ Since $Z$ is ML-random we have $Z \not \in \+ S^k$ for some $k$. So $Z$ is not a dyadic density-one point, as witnessed by  the $\PPI$ class $\cantor - \+ S^k$. 

To define $\seq {\+ S^k} \sN k$ we construct, for each $k, t \in \omega$ and each string $\sss \in U_t$,  clopen sets $\+ A^k_{\sss,t} \sub [\sss]$ given by  strong  indices for finite sets of strings computed from $k, \sss, t$, such that  $\leb (\+ A^k_{\sss,t} )= \tp{-\sssl -k}$ for each $\sss \in U_t$.   We  will let $\+ S^k$ be the union of these  sets over all  $\sss$ and $t$. The  clopen sets  for $k$ and a final string $\sss \in U$ will be disjoint from the    $\PPI$ class $\+ S^k$. Condition  (b) on Madison tests ensures that  during the construction,  a  string $\sss$ can  inherit the clopen sets belonging to   its extensions $\tau$,  without risking that  the $\PPI$ class becomes empty above $\sss$.

\vsp

\n \emph{Construction of clopen sets $\+ A^k_{\sss,t} \sub [\sss]$ for $\sss \in U_t$.}

\n  At stage $0$ no    sets need to be defined  because  $U_0 = \ES$. 
  At stage $t+1$, suppose that    $\sss \in U_{t+1} - U_t$. 
 For each $\tau \succ \sss$ such that  $\tau \in U_t - U_{t+1}$, put $\+ A^k_{\tau, t}$ into an  auxiliary clopen set  $ \wt {\+  A}^k_{\sss, t+1}$. Since $\sss \not \in U_t$, by  condition  (b) on Madison tests,  we have  $\weight (\sss^\prec \cap U_t) \le \tp{-\sssl}$. Inductively we have $\leb (\+ A^k_{\tau,t} )= \tp{-|\tau| -k}$ for each $\tau$, and hence \bc $\leb (\wt {\+ A}^k_{\sss,t+1}) \le  \tp{- \sssl -k}$. \ec Now,  to obtain $ \+ A^k_{\sss , t+1}$ we simply add mass from $[\sss]$ to   $\wt  {\+ A}^k_{\sss , t+1}$ in order to ensure equality as required. 

Let $$\+ S^k_t = \bigcup_{\sss \in U_t} \+ A^k_{\sss,t}.$$ Then $\+ S^k_t \sub \+ S^k_{t+1}$ by condition  (a) on Madison tests. Clearly \bc $\leb \+ S^k_t \le  \tp {-k} \weight (U_t)  \le \tp{-k}$. \ec So $\+ S^k = \bigcup_t \+ S ^k_t$ determines  a ML-test.  Since $Z$ is ML-random, we have  $Z \not \in \+ S^k$ for some $k$. If $\sss \in U$ then by construction  $\leb \+ A^k_{\sss,s} = \tp{-\sssl -k}$ for almost all $s$. Thus  $ \leb_\sss(\+ S^k) \ge \tp{-k}$  as required.  \end{proof}

 We now take  the second step of the argument. We begin with a  remark on  Madison  tests. 
\begin{remark} \label{rem:Doob Madison} {\rm    
	 Consider a computable rational-valued martingale $B$; that is, $B(\sss)$ is a rational uniformly computed (as a single output)   from~$\sss$. Suppose that $c,d$ are rationals, $0< c< d$,  $B(\estring)< c$, and  $B$  oscillates between values less than $c$ and greater than $d$ along a bit sequence $Z$. An \emph{upcrossing} (for these values)  is a pair of strings $\sss \prec \tau$ such that $B(\sss)< c$, $B(\tau)>d$, and $B(\eta) \le d$ for each $\eta$ such that $\sss \preceq \eta \prec \tau$.

Dubins' inequality from probability theory  limits the amount of oscillation a martingale can have; see, for instance, \cite[Exercise 2.14 on pg.\ 238]{Durrett:96}. (Note that this inequality implies a version of  the better-known Doob  upcrossing inequality by taking the sum over all $k$.)    In the restricted setting of martingales on $\strcantor$,  Dubins' inequality shows that for each $k$  \begin{equation} \label{eqn: upcr} \leb\{X \colon \,  \text{there are $k$ upcrossings along} \,  X\} \le (c/d)^k.\end{equation}
See \cite[Lemma 5.8]{Bienvenu.Greenberg.ea:nd} for a proof of this fact using    notation close to the one of  the present paper.

Suppose  now that $2c< d$. We   define a Madison test that     $Z$ fails.  Strings never leave the computable  approximation of the  test, so (a) holds. 

 We put the empty string $\estring$ into $U_1$. If $\sss \in U_{s-1}$,  put into $U_s$ all strings $\eta$ such that $B(\tau) >d$ and $B(\eta) < c$ for some  $\tau \succ \sss$ chosen prefix minimal, and $\eta \succ \tau$ chosen prefix minimal. Let  $U = \bigcup U_s$ (which  is in fact computable). For each $\sigma$, by the   inequality (\ref{eqn: upcr}) localised to $[\sss]$, we have  $\weight (\sss^\prec \cap U) \le    \tp{-\sssl}  \sum_{k\ge 1} (c/d)^k   <  \tp{-\sssl}$, so (b) is satisfied vacuously.  }
  \end{remark}
  
  As   noted in  \cite[Section 5]{Bienvenu.Greenberg.ea:nd}, if $B =\sup B_s$ is a left-c.e.\ martingale where   $\seq{B_s}\sN s$ is a  uniformly computable sequence of martingales,    an upcrossing  apparent at stage $s$ can later disappear because $B(\sss)$ increases.   We will see in the proof of Lemma~\ref{Madison to  MG convergence} that  in this case, the full power of the conditions (a) and (b) is needed to obtain a Madison test from the oscillatory  behaviour of $B$. 

 We use  Remark~\ref{rem:Doob Madison} for an intermediate fact, which is not as obvious as one might expect. 
\begin{lemma} 
Suppose that $Z$ passes each Madison test. Then $Z$ is computably random.
\end{lemma}
\begin{proof}  Suppose $Z$ is not computably random. Then some   computable rational-valued  martingale   $M$ with the savings property succeeds on $Z$ (see \cite[Ex.\ 7.1.14 and Prop.\ 7.3.8]{Nies:book} or \cite{Downey.Hirschfeldt:book}). The proof of \cite[Thm.\ 4.2]{Freer.Kjos.ea:14} turns success of such a martingale    into oscillation of another  computable rational-valued   martingale $B$. Slightly adapting the (arbitrary)  bounds for the oscillation given there, we may assume that  $B$ is as in Remark~\ref{rem:Doob Madison} for $c=2, d=5$:  if $M$ succeeds along $Z$,  then there are infinitely many  upcrossings $\tau \prec \eta \prec Z$,  $B(\tau) < 2$ and $B(\eta) >5$. Therefore $Z$ fails the  Madison test constructed in Remark~\ref{rem:Doob Madison}. 
\end{proof}

We are now ready for the main part of the second step.
\begin{lemma} \label{Madison to  MG convergence} Suppose that $Z$ passes each Madison test. Then every left-c.e.\ martingale $L$ converges along $Z$.  In particular, $Z$ is ML-random. \end{lemma}

\begin{proof}
Let $L$ be a left-c.e.\ martingale. Then  $L(\sss) = \sup_s L_s(\sss)$ where $\seq {L_s}$ is a uniformly computable sequence of martingales, and $L_0 = 0$ and $L_s(\sss) \le L_{s+1}(\sss)$ for each $\sss$ and~$s$.
Since $Z$ is computably random,   $\lim_n L_s(Z \uhr n)$ exists  for each $s$.
If $L$ diverges along $Z$, then $\lim_n L(Z\uhr n)= \infty$ or  there is  a positive $\varepsilon < L(\estring) $ such that  
\[\limsup_n L(Z \uhr n) - \liminf_n L(Z\uhr n) > \varepsilon.\]
Based on this fact  we  define a Madison test that    $Z$ fails. Along with the $U_s$ we define a uniformly computable labelling  function $\gamma_s \colon \, U_s \to \{0, \ldots,s\}$.  If $\lim_n L(Z\uhr n)= \infty$ set $\varepsilon = 1$. The construction is as follows.

\vsp

\leftskip 0.5cm
\n {\it  Let $U_0 = \ES$. For $s>0$ we put  the empty string  $\estring $ into $U_s$ and let $\gamma_s(\estring) = 0$. If already $\sss \in U_s$ with $\gamma_s(\sss) = t$, then we  also put into $U_s$ all  strings $\tau \succ \sss$ that are minimal under the prefix ordering  with $L_s(\tau) - L_t(\tau) > \varepsilon$. Let $\gamma_s(\tau) $ be the least $r$ with $L_r(\tau) - L_t(\tau) > \varepsilon$.   }

\leftskip 0cm

\vsp

\n Note that $\gamma_s(\tau)$ records    the greatest stage $r \le s $ at which $\tau$ entered~$U_r$. Intuitively, this construction  attempts to find   upcrossings  between   values  (arbitrarily  close to)  $ \liminf_n L(Z\uhr n) $ and $ \limsup_n L(Z\uhr n)$. Clearly \bc $\lim_n L_t(Z \uhr n) \le  \liminf_n L(Z\uhr n)$. \ec So, if  a string  $\tau \prec Z$ as above  is sufficiently long,  then   we have an upcrossing of the required kind.

We  verify that  $\seq {U_s}\sN s$ is a Madison test. For condition (a), suppose that  $\tau \in U_s - U_{s+1}$. Let $\sss_0 \prec \sss_1 \prec \ldots \prec \sss_n = \tau$ be the prefixes of $\tau $ in $U_s$. We can  choose a least  $i< n$  such that $\sss_{i+1}$ is no longer the minimal extension of $\sss_i$ at stage $s+1$. Thus there is $\eta$ with $\sss_i \prec \eta \prec \sss_{i+1}$ and $L_{s+1}(\eta) - L_{\gamma_s(\sss_i)}(\eta) > \varepsilon$. Then $\eta \in U_{s+1}$ and $\eta \prec \tau$,  as required.

We  verify condition (b). 
We fix $s$,  and for  $t \le s $ write  \bc $M_t(\eta)=  L_s(\eta ) - L_t(\eta)$. \ec 
Thus $M_t$ is  the increase of $L$ from $t$ to $s$. Note that  $M_t$ is a martingale.
\begin{claim}  \label{cl:tech} For each $\eta \in U_s$, where $\gamma_s( \eta) = r$, we have
\[\tp{- |\eta|}M_r(\eta) \ge \varepsilon \cdot  \weight (U_s \cap \eta^\prec). \]
\end{claim}
\n In particular, if   $\eta = \estring$ then $r=0$; we  obtain that $\weight (U_s)$ is bounded by  a constant $c= L(\estring)\varepsilon^{-1} +1$ (the ``$+1$" is for the empty string in $U_s$), as required.
\n 

For $\sss \in U_s$ and $k \in \NN$, let  $U_s^{\sss}(k)$  be the set of strings properly extending   $\sss$ and at a distance  to  $\sss$ of at most $k$, that is, the set of strings $\tau$ such that   there is $\sss = \sss_0 \prec \ldots \prec \sss_m =\tau$ on $U_s$ with $m \le k$ and $\sigma_{i+1}$ a child (i.e., immediate successor) of $\sss_i$ for each $i<m$.   To establish the claim, we show by induction on $k$ that 
\[\tp{- |\eta|}M_r(\eta) \ge \varepsilon \cdot  \weight (U_s^{\eta}(k)). \]
If $k=0$ then $U_s^{\eta}(k)$ is empty so   the right hand side equals $0$. Now suppose that  $ k>0$. Let $F$  be the  set of  of children of $\eta$ on $U_s$.  For $\tau \in F$ write  $r_\tau = \gamma_s(\tau)$. Then $s \ge r_\tau > r$ by the definition of the function  $\gamma_s$. 
By the inductive hypothesis, we have for each  $\tau \in F$%
\begin{eqnarray}   \label{eqn:taus} \tp{-|\tau|} M_{r}( \tau) & = & \tp{-|\tau|} [(L_{r_\tau}(\tau)- L_r(\tau)) + M_{r_\tau}(\tau)] \\
								& \ge  & \tp{-|\tau|} \cdot \varepsilon + \varepsilon \cdot \weight (U_s ^{ \tau}(k-1)). \nonumber   \end{eqnarray}
								Then, taking the sum over all $\tau \in F$, 
 $$ 	\tp{-|\eta|} M_r(\eta) \ge \sum_{\tau \in F} \tp{-|\tau|}M_r(\tau)	\ge  \varepsilon \cdot \weight (U_s ^{\eta}(k)). $$
 The first inequality holds by  a general fact about for martingales attributed to Kolmogorov (see \cite[7.1.8]{Nies:book}), and uses  that $F$ is  an antichain. For the second inequality we have used (\ref{eqn:taus}) and  that $U_s^{ \eta}(k) =F \cup \bigcup_{\tau \in F} U_s ^{ \tau}(k-1)$. This completes the induction and  shows the claim. 
 
 Now, to obtain (b), suppose that  $\weight (U_s \cap \sss^\prec)  > \tp{- |\sss|}$. We  show that $\sss \in U_s$. Assume otherwise. Let $\eta \prec \sss$ be   in $U_s$ with $|\eta|$ maximal, and let $r = \gamma_s(\eta)$. Let  now $F$   be the set of  prefix minimal extensions  of $\sss$ in $U_s$, and $r_\tau = \gamma_s(\tau)$.   Then $L_{r_\tau}(\tau)- L_r(\tau)> \varepsilon$ for $\tau \in F$. Since  $\tau \in U_s$,  we can apply    Claim~\ref{cl:tech} to $\tau$. 
 We now  argue similar to  the above, but  with   $\sss$ instead of $\eta$,   and using in the last line that $ U_s \cap \sss^\prec = F \cup \bigcup_{\tau \in F} (U_s  \cap \tau^\prec)$:
\begin{eqnarray*} 	\tp{-|\sss|} M_r(\sss) & \ge &  \sum_{\tau \in F} \tp{-|\tau|}M_r(\tau)	\\ 
                              & =  & \sum_{\tau \in F} \tp{-|\tau|} [L_{r_\tau}( \tau ) - L_r(\tau) + M_{r_\tau}(\tau)] \\
                              & \ge &  \sum_{\tau \in F} \tp{-|\tau|} [\varepsilon + \varepsilon \cdot  \weight ( U_s \cap \tau^\prec) ] \\
				& \ge &  \varepsilon \cdot \weight (U_s \cap \sss^\prec). \end{eqnarray*}
  Since  $\weight (U_s \cap \sss^\prec)  > \tp{- |\sss|}$, this implies that $M_r(\sss) > \varepsilon$. Hence  some $\eta'$  with	$\eta \prec \eta' \prec \sss$ is in $U_s$, contrary to the maximality of $\eta$.	
  
  This concludes the verification that $\seq {U_s}\sN s$ is a Madison test. 		As mentioned, for each $r$ there are infinitely many $n$ with $L(Z\uhr n) - L_r(Z \uhr n) > \varepsilon$. This shows that $Z$ fails this test: suppose inductively that we have $\sss \prec Z$ such that there is a least  $r$ with $\sss \in U_t$ for all $t\ge r$ (so that $\gamma_t(\sss) = r$ for all such $t$). Choose $n > \sssl$ for this $r$. Then from some stage on $\tau = Z \uhr n$ is a viable extension of $\sss$, so $\tau$, or some prefix of it that is  longer than $\sss$, is in~$U$.  	
\end{proof} 

\n This concludes our proof of Thm.\ \ref{thm:Madison}. \end{proof}


\section{Differentiability of non-decreasing functions} \label{s:STACS14}
We consider an effective version,  in the sense of computable enumerability, of Lebesgue's theorem~\ref{thm:Leb}   that non-decreasing functions are  almost everywhere differentiable.
Freer, Kjos-Hanssen, Nies and Stephan \cite{Freer.Kjos.ea:14} studied a class of non-decreasing functions   they called \emph{interval-c.e.} They showed (with J.\ Rute) that the continuous interval-c.e.\ functions are precisely the variation functions of computable functions.

\begin{definition} \label{def:intervalce}  
A non-decreasing  function $f\colon [0,1]\to \R$ is \emph{interval-c.e.} if $f(0)=0$, and $f(y)-f(x)$ is a left-c.e.\ real, uniformly in  all rationals $x<y$. 
\end{definition}

We  match an  effective version of Lebesgue's theorem, stated   in terms of interval-c.e.\ functions,  to density randomness. This  result is due to Nies in  the conference paper~\cite{Nies:14}. We give a more detailed proof  here.

\begin{theorem}[\cite{Nies:14}] \label{thm:interval left-c.e. MG and derivative}   $z\in [0,1]$ is density random $\LR$

\hfill      $f'(z)$ exists for each  interval-c.e.\ function $f \colon \, [0,1] \to \RR$.  
\end{theorem}

\lapf  If $z$ is not density random then  by Theorem~\ref{thm:Madison} a  left-c.e.\ martingale $M$ diverges along the binary expansion of $z$. Let $\mu_M$ be the  measure on $[0,1]$ corresponding to $M$, which is given by $\mu [\sss] = \tp{-\sssl} M(\sss)$, and let  $ \cdf_M(x) = \mu_M[0,x)$. Then $\cdf_M$ is interval-c.e.\ and $(\cdf_M)'(z)$ fails to exist.

\vsps
The rest of this section is devoted to proving the implication $\RA$.
This combines       purely analytical arguments with effectiveness considerations.

\subsection{Slopes and martingales}  
\label{ss:slopes_marti} First we need notation and  a few definitions, mostly taken  from   \cite{Brattka.Miller.ea:16} or  \cite{Bienvenu.Hoelzl.ea:12a}.  For a function~$f\colon \, [0,1] \to\R$, the \emph{slope} at a pair $a,b$ of distinct reals in its domain is
\[
S_f(a,b) = \frac{f(a)-f(b)}{a-b}.
\]
For a nontrivial  interval $A$ with endpoints $a,b$, we also write $S_f(A)$ instead of $S_f(a,b)$.

We let $\sss, \tau$ range over (binary) strings.  For such a string  $\sss$,  by $[\sss]$ we denote the closed basic dyadic interval $[0.\sss, 0.\sss + \tp{-\sssl}]$.  The corresponding open basic dyadic interval is denoted $(\sss)$. 


\vsps
\n {\it Derivatives.} If $z$ is in an open neighborhood of the domain of~$f$, the \emph{upper} and \emph{lower derivatives} \label{def_upper_lower_deriv} of $f$ at $z$ are
\[
\ol D f(z)  =  \limsup_{h\ria 0} S_f(z, z+h) \quad  \textnormal{and}    \quad
\underline D f(z)  =  \liminf_{h\ria 0} S_f(z, z+h),
\]
where  $h$ ranges over reals. The derivative $f'(z)$ exists if and only if these values coincide and  are finite.

We will  also consider the upper and lower \emph{pseudo}-derivatives   defined by:   
\begin{align*}
\widetilde Df(x) &= \limsup_{h \to 0^+} \,  \{S_f(a,b)     \mid    \, a\le x \le b \lland\, 0 <  b-a\le h\} , \\
\utilde Df(x) &= \liminf_{h \to 0^+} \, \{S_f(a,b)  \mid   \, a\le x \le b \lland\, 0 <  b-a\le h\}, 
\end{align*}
where $a,b$ range over rationals in $[0,1]$.  We   use them because in our arguments it is often convenient  to consider rational intervals containing~$x$, rather than intervals that have  $x$ as an endpoint. 
 

\begin{remark} \label{rem:BraBra} Brattka et al.\ \cite[after Fact 2.4 ]{Brattka.Miller.ea:16}  verified that \bc  $\ul Df(z) \le \utilde Df(z) \le \widetilde Df(z) \le \ol Df(z)$ \ec for any real~$z\in [0,1]$. 
   To show  $\widetilde  Df(z) \le \ol  Df(z) $, given any real $z$ and rationals $a \le z \le b$ with $a<b$, we have \bc $S_f(a,b)=\frac{b-z}{b-a} S_f(b,z)+\frac{z-a}{b-a} S_f(z,a)\le  \ol  Df(z)$. \ec The   inequality $\ul Df(z) \le \utilde Df(z) $  can be shown in a  similar way. 
   
    If $f$ is nondecreasing one can in fact verify equality, so   the lower and upper pseudo-derivatives of $f$ coincide with the usual lower and upper derivatives. 
   \end{remark}
We will use the subscript $2$ to indicate that all the limit operations  are restricted to the case of  basic dyadic intervals containing   $z$. Thus, 
\begin{eqnarray*}  \widetilde D_2f(x)  & =&  \limsup_{|A| \to 0} \,  \{S_f(A)     \mid    \, x \in A \lland A \text{ is a basic dyadic interval}\},  \\
\utilde D_2f(x)  & =&  \liminf_{|A| \to 0} \,  \{S_f(A)     \mid    \, x \in A \lland A \text{ is a basic dyadic interval}\}.  
 \end{eqnarray*}

\subsection{Porosity and  upper   derivatives}  \label{ss:porous upper}
  We say that a set $\sC\sub\mathbb R$ is \emph{porous at} $z$ via  the constant $\varepsilon >0$ if   there exist  arbitrarily small $\beta>0$  such that $(z-\beta, z+ \beta)$ contains an open interval of length $\varepsilon\beta $ that is disjoint from $\sC$. We say that $\sC$ is \emph{porous at} $z$ if it is porous  at $z$ via some~$\varepsilon>0$. This notion  originated in the work of Denjoy. See for instance \cite[5.8.124]{Bogachev.vol1:07} (but note the typo in the definition there). 

\begin{definition}[\cite{Bienvenu.Hoelzl.ea:12a}]
We call $z$ a \emph{porosity point} if some  effectively closed class to which it belongs is porous at $z$. Otherwise, $z$ is  a \emph{non-porosity point}.
\end{definition}

Clearly, if  $\+C$ is porous at $z$ then  $\ul\varrho(\sC | z)<1$, so $z$ is not   a density-one point.   The converse fails: every Turing incomplete \ML\ random real is    a non-porosity point by~\cite{Bienvenu.Hoelzl.ea:12a}. By \cite{Day.Miller:15} there is such a real such that $\ul\varrho(\sC | z)<1$ for some $\PPI$ class $\+ C$.  We also note that  it is unknown whether a Turing complete \ML\ random real can be a non-porosity point. If not, then the sets of  positive density and non-porosity ML-random reals  coincide.

 We show that  if the  dyadic and full upper/lower derivatives at $z$ are different, then some closed set is porous at $z$. This  extends the  idea in the proof of  Theorem~\ref{th:ML-dyadic-full} due to Khan and Miller. We begin with the easier case of the upper derivative. The other case will be supplied in Subsection~\ref{ss:lower derivative}.
 
 \begin{proposition}\label{pro:interval c.e.}
Let $f \colon \, [0,1] \to \RR$ be interval-c.e. If  $z$ is a   non-porosity point, then  $\widetilde D_2 f(z) = \widetilde Df(z)$ 
\end{proposition}

\begin{proof} Suppose that  $\widetilde D_2 f(z) < p <  \widetilde Df(z)$ for a rational $p$.  Choose $k\in \NN $ such that  $p(1+\tp{-k+1})<  \widetilde Df(z)$.

Let $\sss^* \prec Z$ be any string such that  $\forall \sss \, [ \sss^* \preceq \sss \prec Z    \RA S_f([\sss]) \le p]$. 
It is sufficient to establish the following.
\begin{claim} The closed set 
\[\+ C = [\sss^*] - \bigcup \{ (\sigma) \mid \,    S_f([\sss]) >p\},\label{eqn: def C porous} \]  
which contains  $z$,    	is porous at $z$. \end{claim}
If $f$ is interval-c.e., the function $\sss \to S_f([\sss])$ is a left-c.e.\ martingale. In particular,  $\+ C$  is effectively closed, and  porous at $z$.

The proof of the  claim is purely analytical, and only uses that $f$ is non-decreasing. We show  that there exist arbitrarily large $n$ such that some basic dyadic interval  $[a, \tilde  a]$     of length $\tp{-n-k}$ is disjoint from $\+ C$, and contained in $[z- \tp{-n+2}, z + \tp{-n+2}]$.  In particular,   we can choose $\tp{-k-2}$  as  a porosity constant.

By choice of $k$  there is  an interval $I \ni z$ of  arbitrarily short  positive length such that $  p(1+\tp{-k+1})< S_f(I) $. Let $n$ be such that $\tp{-n+1} > |I| \ge \tp{-n}$. Let $a_0$ be greatest of the form $\ell \tp{-n-k}$, $\ell \in \ZZ$, such that $a_0 <  \min I$. 
 Let $a_v = a_0 + v \tp{-n-k}$. Let $r$ be least such that $a_r \ge \max I$. 

Since $f$ is nondecreasing and $a_r - a_0 \le |I| + \tp{-n-k+1} \le  (1+ \tp{-k+1}) |I|$, we have 
\[ S_f (I)	 \le S_f(a_0, a_r) (1+ \tp{-k+1}  ),\]
 and therefore $S_f(a_0,a_r)>p$.  Since $S_f(a_0,a_r)$ is the average of the slopes $S_f(a_u,a_{u+1})$ for $u<r$, there is a $u<r$ such that  $$S_f(a_u,a_{u+1})>p.$$ Since $(a_u, a_{u+1}) = (\sss)$ for some string $\sss$, this gives the required `hole' in $\+ C$ which is near $z \in I$ and large on the scale of $I$: in  the   definition of porosity   at the beginning of this subsection, let $\beta = \tp{-n+2}$ and note that  we have   $[a_u, a_{u+1}]  \sub [z- \tp{-n+2}, z + \tp{-n+2}]$ because $z \in I$ and $|I| < \tp{-n+1}$. 
  \end{proof}


 	\subsection{Basic dyadic intervals shifted by $1/3$}
	\label{ss:MS}
		  We  will use a  basic `geometric' fact observed,  for instance,  by   Morayne and Solecki~\cite{Morayne.Solecki:89}. 
	For $m \in \NN$ let $\+D _m $ be the collection of intervals of the form $$[k \tp{-m}, (k+1)\tp{-m}]$$ where $k \in \ZZ$. Let $  \widehat {\+ D}_m$ be the set of  intervals $(1/3)  +I $ where $I \in \+ D_m$. 

	\begin{lemma} \label{fact:geom} Let $m \ge 1$.  If  $I \in \+ D_m$ and $J \in \widehat {\+ D}_m$, then the distance between an  endpoint of $I$ and an endpoint of $J$ is at least $1/(3 \cdot 2^m)$.
	\end{lemma}
	To see this, assume  that  $|k \tp{-m} - ( p \tp{-m} +1/3) | < 1/(3 \cdot 2^m)$. This yields $|3k-3p-2^m|/ (3 \cdot2^m) < 1/(3 \cdot 2^m)$, and hence $3| 2^m$, a contradiction.

In order to  apply Lemma~\ref{fact:geom},   we may   need values   of nondecreasing functions $f \colon \, [0,1] \to \RR$   at endpoints of any such intervals, which may lie outside $[0,1]$. So  we think of   $f$ as extended to $[-1, 2] $ via $f(x) = f(0)$ for $-1\le x< 0$ and $f(y) = f(1)$ for $1 < y\le 2$.  Being  interval-c.e.\  is  preserved by this because it suffices to determine the values of the function   at rationals.   

\subsection{Porosity and  lower   derivatives} \label{ss:lower derivative}
 We complete   the proof of the implication ``$\RA$'' in Theorem~\ref{thm:interval left-c.e. MG and derivative}.  We  may assume that  $z> 1/2$.  
 Note that  $z-1/3$ is   a ML-random density-one point, hence a dyadic density-one point.  In particular, both  $z$ and $z-1/3$ are    non-porosity points.
Also,   Theorem~\ref{thm:Madison},  all left c.e.\ martingales converge  on the binary expansions of the reals $z$ and  $z-1/3$.

	 Let $M=M_f $ be the   left-c.e.\   martingale given by  $\sss \to  \ S_f([\sss])$. 
Then $M$ converges on $z$  (recall that we write $M(z)$ for the limit).  Thus  $\utilde D_2f(z)= \widetilde D_2f(z) =M(z)$.

Let $\widehat f(x) = f(x+1/3)$, and let $\hat M = M_{\widehat f}$.  Then  $\hat M$ converges on $z-1/3$.
 \begin{claim}  $M(z) = \hat M(z-1/3)$.   \end{claim} 
If   $M(z)  < \hat M(z-1/3)$ then $\widetilde D_2 f(z) <  \widetilde Df(z)$.
However, $z$ is a non-porosity point, so this contradicts Proposition~\ref{pro:interval c.e.}. 
	If   $\hat M(z-1/3)  < M(z) $  we argue similarly using that $z-1/3$ is a non-porosity point. This establishes the claim.

  We have already shown that       $\utilde D_2f(z) = \widetilde D_2f(z)= \widetilde Df(z)$, so  to complete the proof of ``$\RA$'' in  Theorem~\ref{thm:interval left-c.e. MG and derivative}, it suffices  to show that 
	$\utilde Df(z)= \utilde D_2f(z)$.   Then, since $f$ is nondecreasing,  $f'(z)$ exists by  Remark~\ref{rem:BraBra}.

Assume for a contradiction that if $\utilde Df(z) <  \utilde D_2f(z)$. We will show that  one of $z$,  $z-1/3$	is a porosity point.   First we define porosity in Cantor space.

\begin{definition} \label{def: dyad porous} For a closed set $\sC \sub \cantor$, we say that  $\sC$ is porous at  $Y \in \sC$ if  there is   $r \in \NN$ as follows: there exists arbitrarily large $m$ such that \bc $ \sC \cap [(Y\uhr m) \ape \tau] = \ES$ for some $\tau$ of length $r$.  \ec  \end{definition} 
\n Clearly this implies that $\sC$ viewed as a subclass of $[0,1]$ is porous at $0.Y$ (now ``holes'' on both sides of $0.Y$ are  allowed). We will actually define $\Pi^0_1$ classes $\+ E$ and $	\widehat {\+ E}$   in Cantor space such that $\+ E$ is porous at   the binary expansion  of  $z$,  or $\widehat {\+ E}$ is porous at   the binary expansion  of  $z-1/3$.	

We employ a  method similar to  the one   in Subsection~\ref{ss:porous upper}, but  now take into account both dyadic intervals, and dyadic intervals shifted by $1/3$ of the same length. Recall that  $\utilde D_2f(z) = M(z)$.  

We can choose rationals $p,q$ such that 
	\[\utilde D f(z) < p < q < M(z) = \hat M(z-1/3).\]
	Let $k\in \NN$ be such that $p< q(1- \tp{-k+1})$. Let $u,v$ be rationals such that 
	\bc $ q< u <  M(z)  <v$ and $v-u\le \tp{-k-3}(u-q)$. \ec
	
	Recalling the notation in Subsection~\ref{ss:MS}, let $n^* \in \NN$ be such that for each $n \ge n^*$ and any interval $A\in \+ D_n \cup  \widehat {\+ D}_n$ containing $z$, we have $S_f(A) \ge u$.
	Let 
	\begin{eqnarray*} \+ E &=&  \{ X \in \cantor \colon \, \fa n \ge n^* M(X\uhr n)\le v \}\\
		\widehat {\+ E} &=&  \{ W \in \cantor \colon \, \fa n \ge n^* \hat M(W\uhr n) \le v \} \end{eqnarray*}
		Since $f$ is interval-c.e.,  $M$ and $\hat M$ are left-c.e.\ martingales, so these classes are effectively closed.

		Let  $Z$ be the bit sequence such that   $z = 0.Z$.   By the choice of $n^*$ we  have  $Z \in \+ E $. Let  $Y$ be   the bit sequence such that $0.Y= z-1/3$. We have   $Y \in  \widehat {\+ E}$.

Consider   an interval  $I \ni z$   of  positive length $\le \tp{-n^*-3}$ such that $S_f(I) \le p$. Let $n$ be such that $\tp{-n+1} > |I| \ge \tp{-n}$. Let $a_0$  be least of the form $w \tp{-n-k}$  where $w \in \ZZ$, such that $a_0  \ge  \min (I)$.   Similarly, let 
 $b_0$  be least of the form $w \tp{-n-k}+1/3 $  such that $b_0  \ge  \min (I)$. 
	 Let \bc  $a_i = a_0 + i \tp{-n-k}$ and $b_j = b_0 + j \tp{-n-k}$. \ec Let $r,s$ be greatest  such that $a_r \le \max (I)$ and $b_s \le \max(I)$. 

Since $f$ is nondecreasing and \bc  $a_r - a_0 \ge  |I| - \tp{-n-k+1} \ge  (1- \tp{-k+1}) |I|$,  \ec we have 
$S_f (I)	 \ge S_f(a_0, a_r) (1- \tp{-k+1}  )$,
	 and therefore $S_f(a_0,a_r)< q$. Then there is an $i<r$ such that  $S_f(a_i,a_{i+1})< q$. Similarly, there is $j< s$ such that $S_f(b_j,b_{j+1})< q$.
	
	\begin{claim}
	One of the following is true.

 {\rm (i)} $z, a_i,a_{i+1} $ are   all  contained in a single interval   from $\+  D_{n-3}$. 
	
 {\rm (ii)} $z, b_j,b_{j+1} $   are  all  contained in a single   interval   from $\widehat {\+ D}_{n-3}$. 	\end{claim}
For suppose that (i) fails. Then there  is  an  endpoint of an interval $A\in \+ D_{n-3}$ (that is, a number of the form $w\tp{-n+3}$ with $w\in \ZZ$) between $\min (z, a_i) $ and $\max (z, a_{i+1})$. Note that $\min (z, a_i) $ and $\max (z, a_{i+1})$ are  in $I$. By Fact~\ref{fact:geom} and since  $|I| < \tp{-n+1}$,  there can be no endpoint of an interval $\hat A \in \widehat {\+ D}_{n-3}$ in $I$. Then, since $b_j, b_{j+1} \in I $,  (ii) holds. This establishes  the claim.

Suppose $I$ is an interval as above and  $\tp{-n+1} > |I| \ge \tp{-n}$, where $n \ge n^*+3$. Let $\eta = Z \uhr {n-3}$ and $\hat \eta = Y \uhr {n-3}$. 
	
	If (i) holds for this $I$ then there is  a string $\alpha$ of length $k+3$ (where $[\eta \alpha]=[a_i, a_{i+1}]$) such that $M( \eta \alpha) < q$. So by the choice of $q< u< v$ and since $M(\eta) \ge u$  there is $\beta$  of length $k+3$  such that $M(\eta \beta)> v$. (The decrease along $\eta \aaa$ of the martingale $M$ must be balanced by an increase along  some $\eta \beta$.) This yields  a   ``hole'' in $\+ E$,  large and near $Z$ on the scale of $I$, as required  for the porosity of $\+ E$ at $Z$; in the notation of the Definition~\ref{def: dyad porous}  above,  $\+ E$ is porous at $Z$ via $m = |\eta|$ and $r = k+3$.

		Similarly, if (ii) holds for this $I$, then there is a string $\alpha$ of length $k+3$ (where $[\hat \eta \alpha]=[b_j, b_{j+1}]$) such that $M(\hat \eta \alpha) < q$. So by the choice of $q< u< v$ and since $\hat M(\hat \eta)\ge u$,  there is a string  $\beta$  of length $k+3$  such that $\hat M(\hat \eta \beta)> v$.  This yields  a   hole in  $\widehat {\+ E}$,  large and near $Y$ on the scale of $I$, as  required  for the porosity of $\widehat {\+ E}$ at $Y$. 

 Thus, if case (i) applies for arbitrarily short intervals $I$, then $\+ E$ is porous at $Z$, whence $z$ is a porosity point. Otherwise (ii) applies   for intervals below a certain length. Then   $\widehat {\+ E}$ is porous at $Y$, whence $z-1/3$ is a porosity point.  Both cases are contradictory.  This concludes the proof of Theorem~\ref{thm:interval left-c.e. MG and derivative}.

Nies~\cite{Nies:14}  also uses porosity for an  effective version of Lebesgue's theorem~\ref{thm:Leb} in the setting of polynomial time computable functions and martingales.  The proof can be easily adapted to the original setting of computable functions and martingales, thereby providing a simpler proof of the main result in Brattka et al., Theorem~\ref{thm: comp dyadic diff}.


\section{Lebesgue differentiation  theorem}
This section is centred around  an effective version, in the c.e.\ setting, of another result obtained by Lebesgue  in 1904  \cite{Lebesgue:1904}. 

\begin{definition} \label{Leb point}
Given an   integrable non-negative function $g$ on $[0,1]$, a point $z$ in the domain of $g$ is called a \emph{weak Lebesgue point} of $g$ if
\bc $ \lim_{Q  \to z}  \frac{1}{\lambda(Q)}\int_{Q} g$  \ec  exists, where $Q$ ranges over open intervals  containing $z$ with length  $\lambda(Q)$ tending to $0$;
 $z$ is called a \emph{Lebesgue point}  of $g$  if this value equals $g(z)$.    \end{definition}
 
 We note that also a variant of this definition can   be found  in the literature, where  $Q$ is centred at $z$. This is in fact equivalent to the definition given here; see for instance \cite[Thm.\ 7.10]{Rudin:87}
 \begin{thm}[Lebesgue \cite{Lebesgue:1904}] Suppose $g$ is an integrable function on $[0,1]$. Then almost every $z\in [0,1]$ is a Lebesgue point of $g$.\end{thm} 

Equivalently,  the function $f(z) = \int_{[0,z]} g d\lambda$ is differentiable  at almost every $z$, and  $f'(z) = g(z)$.   

Several years later, Lebesgue   \cite{Lebesgue:1910}  extended this result to higher dimensions; the variable $Q$ now ranges over open cubes containing~$z$.

\subsection{Effective Lebesgue Differentiation Theorem via $L_1$-computability}

 \label{ss:Leb-L1-comp}

Pathak, Rojas and Simpson~\cite[Theorem 3.15]{Pathak.Rojas.ea:12} studied an effective version  of Lebesgue's theorem, where the given function is $L_1$-computable, as defined in \cite{Pour-El.Richards:89} (or see~\cite[Def.\ 2.6]{Pathak.Rojas.ea:12}). They  showed that  

\vsps
\n $z$ is Schnorr random  $\LR$  

\hfill $z$ is a weak Lebesgue point of each $L_1$-computable function. 

\vsps

\n The   implication   ``$\RA$'' was independently obtained in  \cite[Thm.\ 5.1]{Freer.Kjos.ea:14}.  Using this result, we observe that 
if  a $\PPI$ class has computable measure,    it has density $1$ at every Schnorr random member.  
\begin{proposition} Let $\+ P \sub [0,1]$ be an  effectively closed set such that  $\leb \+ P$ is computable. Let  $z \in \+ P$ be   a Schnorr  random real. Then $ \ul  \varrho(\+ P \mid z) =1$. \end{proposition}

\begin{proof}  
 Let $P = \bigcap_s P_s$ for  a computable sequence $\seq{P_s}$  of finite unions of closed intervals.   There is a computable function  $g$  such that $\leb (P_{g(n)}-P) \le \tp{-n}$. Hence the characteristic function $1_P$ is $L_1$-computable. Now  by \cite[Theorem 3.15]{Pathak.Rojas.ea:12} or  \cite[Thm.\ 5.1]{Freer.Kjos.ea:14}, the density of $\+ P$ at $z$ exists, that is $\ul \varrho(P\mid z) = \ol \varrho(P\mid z)$. 
 
 The binary expansion $Z$ of the real $z$ is Kurtz random, so by Proposition~\ref{prop:PC random upper density}(ii) we have $\ol \varrho_2(P\mid Z)=1$. Therefore $\ul \varrho(P\mid z) =1$.
\end{proof}

\subsection{Dyadic Lebesgue points  and  integral tests}
Recall that an open    {basic dyadic interval} in $[0,1]$ has the form $(i\tp{-n}, (i+1)\tp{-n})$ where $i < \tp{n}$. If  a string $\sss$ of length $n$  is the binary expansion of $i$,   we also write $(\sss)$ for this interval. 
We say that $z$ is a (weak) \emph{dyadic Lebesgue point}
if the limit in Definition~\ref{Leb point} exists when~$Q$ is restricted to  open basic dyadic intervals, 
 
As usual  let $ \ol { \mathbb{R}}=\mathbb R \cup \{-\infty, \infty\}$. 
For a function $f:  [0,1]\to \ol { \mathbb{R}}$ and $z\in[0,1]$,
let 
\[E(f,\sigma)=\frac{\int_{(\sigma)}f\ d\leb}{2^{-n}}.\]
Then, $z$ is a dyadic Lebesgue point
iff $\lim_n E(f,Z\uhr n)=f(z)$ where $z= 0.Z$.

Recall from the introduction that a function $g:[0,1]\to \real \cup \{\infty\}$ is lower semi-computable if $f^{-1}(\{z \colon \, z >  q \})$ is effectively open, uniformly in a rational~$q$.   (This is an  effective version of   lower semicontinuity.)
It is  well-known that such functions can be used to characterise \ML\ randomness; see for instance  Li and  Vit\'anyi \cite[Subsection 4.5.6]{Li.Vitanyi:93}.

\begin{definition}  An \emph{integral test}  is a non-negative   lower semi-computable  function $g: [0,1]\to \ol {\real}$ such that $\int  g\ d\leb <\infty$.  \end{definition}
 \begin{theorem}[Levin] \label{MLIntegral}
 A real $z$ is Martin-Löf random if and only if  
 
 \hfill $g(z)<\infty$ for each integral test $g$. 
\end{theorem}

Note that 
if   $f$ is an integral test,   the function $\sss \mapsto E(f,\sss)$ is a left-c.e.\ martingale. Since $f$ is integrable,  $f^{-1} ( \{\infty\})$ is  a null set.   

In  Definition \ref{Leb point} of [weak] Lebesgue points, we   allow  functions $g$  that can take  the value $\infty$.  For $z$ to be a (weak) Lebesgue point, the limit as the intervals  approach $z$ is required to be   finite.  First we show   that  for  an integral test $g$, the dyadic versions of the  weak and strong conditions in Def.\ \ref{Leb point} coincide at   a ML-random real $z$.
\begin{lemma}\label{lem:weak-l}
Let $g$ be an  integral test, and let $z$ be a  \ML \ random real. If $z$  is a dyadic \emph{weak} Lebesgue point of    $g$, then $z$ is in fact a dyadic Lebesgue point of  $g$.  
\end{lemma}
\begin{proof}
Let $\seq{g_s} \sN s$ be an increasing computable sequence of step functions
with dyadic points of discontinuity and rational values
such that $\sup g_s(z)=g(z)$ for each dyadic irrational
(see Miyabe~\cite[Lemmas 4.6, 4.8]{Miyabe:12}, a variant of \cite[Prop.\ 2]{Vyugin:98}).
Then, there is a non-decreasing computable function $u\colon \, \NN \to \NN$ such that for each $\sigma$ with  $|\sigma|\ge u(s)$
\bc $E(g_s,\sigma)=E(g_s,\sigma0)=E(g_s,\sigma1)$. \ec 
Unless $z$ is a dyadic rational,
we have
$g_t(z)  =\lim_n E(g_t,Z\uhr n)$,
where, as usual, $0.Z$ is the binary expansion of $z$.

By hypothesis,  $\lim_n E(g,Z\uhr n)=:r$ exists. Clearly
$g(z) \le r$, because for each $t$
\begin{equation} \label{eqn: approx below}  g_t(z) =  \lim_n E(g_t,Z\uhr n)  \le   \lim_n E(g, Z \uhr n). \end{equation}

 Suppose for  a contradiction  that $g(z)< r$, and 
let $q$ be a rational number such that $g(z)<q<r$.
We build an integral test $h$ such that $h(z)=\infty$, which contradicts our assumption that $z$ is ML-random.
To do so, we  define a  uniformly c.e.\ sequence of sets  $S_n\subseteq\strcantor \times\omega$.
Let $S_0=\{(\estring,0)\}$.
Suppose now that  $n\ge1$ and $S_{n-1}$ has been defined. Uniformly in   $(\sigma,s)\in S_{n-1}$,   let $B \sub  \strcantor $ be  a c.e.\ antichain  of strings of length $\ge u(s)$ such that 
\[ \Opcl B = \Opcl  {\{ \tau \succ \sigma \colon \, |\tau| \ge u(s) \lland  \ex t \, E(g_t, \tau ) > q \}}. \]
For each $\tau \in B$ let $t>s$ be the least corresponding stage and put $\la \tau, t \ra $ into $S_n$.

Let $\mathbf{1}_A$ denote the characteristic function of a set $A$. For each $(\tau,t)\in S_n$,
let
\[h_\tau=(q-E(g_s,\tau))\mathbf{1}_{[\tau]}\]
where $(\sigma,s)\in S_{n-1}$ and $\sigma\prec\tau$.
We define $h$ by
\[h=\sum_n\sum_{(\tau,t)\in S_n}h_\tau.\]

We aim to show that $h$ is an integral test and $h(z)=\infty$. So $z$ is not ML-random  contrary to  our assumption.

To see that $h$ is an integral test,   note that  $h$ is lower semicomputable.
So  it suffices to show that, for every $N$,
\[\sum_{n=0}^N\sum_{(\tau,t)\in S_n}\int h_\tau\ d\leb\le\int g\ d\leb<\infty.\]

If $(\tau, t)\in S_{n}$, for $n>0$,  let $(\sigma_{\tau,t}, s_{\tau,t})\in S_{n-1}$ be the corresponding element for which $(\tau,t)$ is enumerated into $S_{n+1}$.

Notice that
\[\int h_\tau\ d\leb\le(E(g_t,\tau)-E(g_{s_{\tau,t}},\tau))2^{-|\tau|}
=\int_{[\tau]}(g_t-g_{s_{\tau,t}})d\leb,\]
whence
\[\sum_{(\tau,t)\in S_n}\int h_\tau d\leb
\le\sum_{(\tau,t)\in S_n}\int_{[\tau]}(g-g_{s_{\tau,t}})d\leb
\le\sum_{(\sigma,s)\in S_{n-1}}\int_{[\sigma]}(g-g_s)d\leb.\]
Then, in case $N \ge 2$,
\begin{align*}
\sum_{n=N-1}^N\sum_{(\tau,t)\in S_n}\int h_\tau d\leb
\le&\sum_{(\tau,t)\in S_{N-1}}\int_{[\tau]}(g-g_t)d\leb
+\sum_{(\tau,t)\in S_{N-1}}\int_{[\tau]}(g_t-g_{s_{\tau,t}})d\leb\\
\le&\sum_{(\tau,t)\in S_{N-1}}\int_{[\sigma]}(g-g_{s_{{\tau,t}}})d\leb \\
\le&\sum_{(\tau,t)\in S_{N-2}}\int_{[\sigma]}(g-g_t)d\leb. \ \ \ 
\end{align*}
By iterating  this argument for sums starting at $N-2, N-3, \ldots, 2$, we have
\[\sum_{n=0}^N\sum_{(\tau,t)\in S_n}\int h_\tau d\leb
\le\sum_{(\tau,t)\in S_0}\int_{[\tau]}(g-g_s)d\leb=\int g\ d\leb<\infty.\]

Finally, since $\lim_n E(g,Z\uhr n)=r>q$,  for each $n$
there exists $(\tau_n,t_n)\in S_n$ such that $\tau_n\prec z$.
Then
\[h(z)=\sum_n (q-E(g_s,\tau_n))
\ge\sum_n(q-g(z))=\infty.\]
\end{proof}

\begin{remark} The proofs of Lemma \ref{lem:weak-l} and of  Theorem \ref{thm:Madison}    are related. In the notation of Lemma \ref{lem:weak-l}, we have a left-c.e.\ martingale  $L(\sss) = E(g, \sss)$ and uniformly computable martingales  $L_s( \sss)= E(g_s, \sss)$ so that $L(\sss) = \sup_s L_s(\sss)$. By definition of the $g_s$ as dyadic step functions, we have a computable function $u$ on $\NN$ such that $L_s (\tau)= L_s(\tau\uhr {u(s)})$ whenever $|\tau| \ge u(s)$. Let us say that  a left-c.e.\ martingale $L$ of this kind is  \emph{stationary in approximation}.   The obvious inequality  

\[ \sup_s L(Z\uhr {u(s) }) \le \liminf_n L(Z\uhr n) \]
corresponds to $f(z) \le r$ before~(\ref{eqn: approx below}).  
If $Z$ is density random then $\lim_n L(Z\uhr n)$ exists, and equals $\sup_s L(Z\uhr {u(s) }) $ by an argument similar to the one in the proof of Lemma \ref{lem:weak-l}.
\end{remark}

\subsection{Effective Lebesgue Differentiation Theorem via  lower semi-computability} \label{ss:Leb-lowersemi-comp}

We show that density randomness is the same as being a Lebesgue point of each integral test. We use as a basic fact: if $g$ is a non-negative integrable function, then $\sss \to E(g, \sss)$ is a martingale. By definition, $z$ is a weak dyadic  Lebesgue point of $g$ iff this martingale converges along~$Z$.
\begin{theorem}\label{th:density-lebesgue}
The following are equivalent for $z\in[0,1]$:
\begin{enumerate}
\item $z$ is density random.
\item $z$ is a dyadic Lebesgue point of each integral test.
\item $z$ is a Lebesgue point of each integral test.
\end{enumerate}
\end{theorem}

 We could equivalently formulate (ii) and (iii) in terms of integrable lower semicomputable functions, rather than the seemingly  more restricted integral tests. For, any lower semicontinuous function on a compact domain is bounded below. So any integrable lower semicomputable function on $[0,1]$  becomes an integral test after adding a constant. 

\begin{proof} (ii) $\Rightarrow$ (i).   By definition  $g(z)$ is finite for each integral test $g$,
whence $z$ is ML-random.

Let $\sC$ be a $\Pi^0_1$ class containing $z$.
Clearly the function $g= 1 - \mathbf 1_\sC$ is an integral test.
Since $z$ is a Lebesgue point for $g$,
$\sC$ has dyadic  density one at~$z$.  Then, by Theorem~\ref{th:ML-dyadic-full}, $z$ is a density-one point.

\vsps

 \n  (i) $\Rightarrow$ (ii).  Let $g$ be an integral test.
Then  $\sss \to E(g,\sss)$    is a left-c.e.\ martingale.
By Theorem \ref{thm:Madison},
$\lim_n E(g,Z\upr n)$ exists,
whence $z$ is a dyadic weak Lebesgue point for $g$.
By Lemma \ref{lem:weak-l},
$z$ is a dyadic Lebesgue point for $g$.

\vsps

\n  (ii) $\RA$ (iii).  
Let $g$ be an integral test.
The function  $f(x)=\int_{[0,x]}g\ d\leb$ is interval-c.e.\ by the aforementioned result of Miyabe~\cite[Lemmas 4.6, 4.8]{Miyabe:12}.
The real  $z$ is density random  by   (ii)$\to$(i), so $f'(z)$ exists by  Theorem~\ref{thm:interval left-c.e. MG and derivative}.
In particular, $\lim_{Q\to z}  {\leb(Q)}^{-1}  {\int_Q f\ d\leb}$ exists
and equals  $\lim_n  2^n \int_{(Z\upr n)}f\ d\leb=f(z)$.
Hence, $z$ is a Lebesgue point for $g$.

The implication  (iii) $\Rightarrow$ (ii) holds by definition.
\end{proof}

\section{Birkhoff's   theorem}
 
We give an effective version, in the c.e.\ setting,  of Birkhoff's Theorem~\ref{thm:Birkh}.   Franklin and Towsner \cite{Franklin.Towsner:14}   considered the case of a not necessarily ergodic measure-preserving operator $T$ on Cantor space $\cantor$ with the uniform measure,  and a lower semicomputable function $f$.    They showed that the   limit of the averages  in the sense of Theorem~\ref{thm:Birkh} exists  for each weakly 2-random point $z$. Under an additional, hypothetical assumption, in \cite[Thm.\ 5.6]{Franklin.Towsner:14}  they were able to obtain convergence on the weaker assumption that $z$ is balanced random in the sense of  \cite{Figueira.Hirschfeldt.ea:15}.

We work in the more general setting of Cantor space $\cantor$ with a computable probability measure $\mu$. That is,    $\mu [\sss]$ is a left-c.e.\  real uniformly in a string $\sss$. For background see  Hoyrup and Rojas \cite{Hoyrup.Rojas:09}

%

 Bienvenu, Greenberg, \Kuc, Nies, and Turetsky \cite[Def. 2.5]{Bienvenu.Greenberg.ea:nd} introduced a randomness notion that implies density randomness.     A   \emph{left-c.e.\ bounded test} over $\mu$ is a nested sequence  $\seq{\+ V_n}$ of uniformly  $\Sigma^0_1$ classes such that for some  computable sequence of rationals $\seq{\beta_n}$ and  $\beta = \sup_n \beta_n \le 1 $ we have $\mu(\+ V_n)\le \beta-\beta_n$ for all $n$. $Z$ fails this test if $Z \in \bigcap_n \+ V_n$.   $Z$ is $\mu$-\emph{Oberwolfach  (OW)} random if it passes each left-c.e.\ bounded test.   
 
Let  $\mu = \leb$ be the uniform measure;  it  is known that balanced randomness  in the sense of \cite{Figueira.Hirschfeldt.ea:15} implies OW randomness, which implies density randomness. The converse of the first implication fails,   as noted in \cite{Bienvenu.Greenberg.ea:nd}: some  low ML-random  is not   balanced random  \cite{Figueira.Hirschfeldt.ea:15}; on the other hand, any such set is OW random.  It is unknown whether the converse of the second implication holds.

 In the following let $\mu$ be a probability measure on $\cantor$  which is computable  in the strong sense of \cite{Vyugin:98} that $\mu [\sss]$ is a computable real uniformly in a string~$\sss$. Note that this is equivalent to the weaker condition above that $\mu [\sss]$ is uniformly left-c.e.,  in the case that  the boundary of any open set is a null set.   
 \begin{theorem}   
 Let  $T$ be   a  computable measure preserving operator  on   $(\cantor, \mu)$.  Let $f$ be    a non-negative   integrable lower semicomputable function on $X$. Let $A_nf(x) $ be the usual ergodic average  \[ \frac 1 n \sum_{i<n} f \circ T^i(x).   \]  For  every $\mu$-Oberwolfach random point $z \in X$,  $\lim_n A_n f(z)$ exists. \end{theorem}
 Note that    we do not assume  that the operator $T$ is total. However, being  measure preserving, its domain is  conull. Since  $T$ is computable, the domain is also $\PI 2$. So $T(x)$ is defined whenever $x$ is $\mu$-Kurtz random, namely, $x$ is in no $\PI 1$ class $\+ P$ with $\mu \+ P =0$.
 \begin{proof} By V'yugin~\cite[Prop.\ 2]{Vyugin:98}, we have $f(z) = \sup_t f_t(z)$ for every $z \in X$ that is Kurtz  random  w.r.t.\ $\mu$, where $\seq {f_t}$ is a computable non-decreasing sequence of simple functions (namely, there is a partition of $\cantor$ into finitely many clopen sets such that  $f_t$ has constant rational value on each of them).   
 
Since simple functions are computable,  by the main result of V'yugin~\cite[Thm.\ 2]{Vyugin:98}, $\lim_n A_n f_t(x)$ exists for each $t$ and each ML-random point~$x$. By the maximal ergodic inequality (see e.g.\  Krengel \cite[Cor.\ 2.2]{Krengel:85}), for each non-negative integrable function $g$ and each $r> 0$, we have
   \[ \mu \{x \colon \, \ex n \, A_n g(x) > r\} < \frac 1 r \int g   d\mu.  \]
 Since $z$ is weakly random, for each $n$ the value $A_n f(z)$ exists.   Thus, if $\lim_n A_n f(z)$ fails to exists, there are reals $a<b$ such that $A_n f(z) < a$ for infinitely many $n$, and $A_n f(z) >b$ for infinitely many $n$.   
 
 Let \[ \+ V_t = \{x \colon \, \ex k \, A_k (f-f_t) > b-a\}	. \]
 Then $\seq{\+ V_t}\sN t$ is  a sequence of uniformly  $\Sigma^0_1$   open sets in $X$ with $\+ V_t \supseteq \+ V_{t+1}$.  By the maximal ergodic inequality we have $\mu {\+ V_t} \le 1/(b-a) \int (f-f_t) d\mu$. Finally, $\lim_n A_n f_t(z)$ exists for each $t$, and $\lim_n A_n f_t(z) \le a$. Therefore $z \in \bigcap_t \+ V_t$.    \end{proof}

\section{Density-one points for $\Pi_n^0$ classes and $\Sigma^1_1$ classes} \label{Pin} 
In this section we work again in the setting of Cantor space. So far we have looked at the density of $\PPI$ classes at points. Now we will consider classes of higher descriptional complexity.  Firstly,  we look at $\Pi^0_n$ classes. It turns out that if $Z$ is density random relative to $\emptyset^{(n-1)}$,     then  each  $\Pi_n^0$ class has density $1$ at   $Z$.  

 Thereafter we consider the density of  $\VI$ classes at $Z$. This  complexity  forms a natural  bound for our investigation because $\VI$ classes are measurable (Lusin;  see e.g.\  \cite[Thm.\ 9.1.9]{Nies:book}), which is no longer true within ZFC  for more complex classes.

\subsection{Density of $\Pi^0_n$ classes at a real}  Recall that $Z$ is $n$-random if $Z$ is ML-random relative to $\emptyset^{(n-1)}$. By a $\Pi^{0,X}_1$ class we mean a $\PPI$ class relative to $X$. Every $\Pi_{1}^{0,\emptyset^{(n-1)}}$ class  is $\Pi^0_n$. We show that  for an $n$-random $Z$,  it is sufficient to consider $\Pi_1^{0,\emptyset^{(n-1)}}$ classes  in order to obtain that every $\Pi^0_n$ class has density one at $Z$.  To do so, we rely on  a   lemma about the approximation in terms of measure  of $\Pi^0_n$ classes by $\Pi_{1}^{0,\emptyset^{(n-1)}}$ subclasses. This  can be seen as an effective form of regularity for Lebesgue measure. See \cite[Thm.\ 6.8.3]{Downey.Hirschfeldt:book} for a recent write-up of the proof.

\begin{lemma}[Kurtz \cite{Kurtz:81}, Kautz \cite{Kautz:91}]\label{KK}
From an index of a $\Pi_n^0$ class $P$ and $q\in \mathbb{Q}^+$,   $\emptyset^{(n-1)}$ can compute an index of a $\Pi_1^{0,\emptyset^{(n-1)}}$ class $V\subseteq P$ such that $\lambda(P)-\lambda(V)<q$.
\end{lemma}

\begin{theorem}
Suppose $n\geq 1$ and $Z\in \cantor$ is density random relative to $\emptyset^{(n-1)}$.  Let $P$ be $\Pi_n^0$ class such that  $Z\in P$. Then $ \ul \varrho_2(P|Z)=1$.
\end{theorem}
\begin{proof}

Let    $P=\bigcap_s U_s$ where $\langle U_s: s\in \omega\rangle$ is a nested sequence of uniformly  $\Sigma_{n-1}^0$ classes. It suffices to  show that there exists a $\Pi_1^{0,\emptyset^{(n-1)}}$ class $Q\subseteq P$ such that $Z\in Q$.

We   define a Solovay test relative to $\emptyset^{(n-1)}$. By Lemma \ref{KK}, effectively in $\emptyset^{(n-1)}$  we obtain an  index of a  $\Pi_1^{0,\emptyset^{(n-1)}}$ class $Q_s\subseteq U_s$ such that 
\[\lambda(U_s)-\lambda(Q_s)<2^{-n}.\]  The sequence of uniformly $\Sigma_1^{0,\emptyset^{(n-1)}}$ classes
\[\langle U_s\backslash  Q_s: s\in \NN\rangle\]   is a Solovay test relative to $\emptyset^{(n-1)}$ since $\lambda(U_s \backslash Q_s)\leq 2^{-s}$. Notice $Z\in P\subseteq U_s$ for each $s\in \NN$. Since $Z$ is Martin-Löf random relative to $\emptyset^{(n-1)}$, there exists $k\in \NN$ such that for all $j\geq k$, $Z\in Q_j$.  Since $\langle Q_j: j\geq k\rangle$ is a uniform sequence of $\Pi_1^{0,\emptyset^{(n-1)}}$ classes,    $V = \bigcap_{j \geq k} Q_j$ is itself  a $\Pi_1^{0,\emptyset^{(n-1)}}$ class. Also  $V\subseteq \bigcap_{i\in \NN}U_i=P$ because $Q_j\subseteq U_j$. We have found a $\Pi_1^{0,\emptyset^{(n-1)}}$ class $V\subseteq P$ that contains~$Z$.
\end{proof}
Relativizing Theorem~\ref{thm:Madison} to $\emptyset^{(n-1)}$ we obtain:
\begin{cor}
An $n-$random set $Z$ is a density one point for $\Pi_n^0$ classes if and only if every left-$\emptyset^{(n-1)}$-c.e. martingale converges along $Z$.
\end{cor}

\subsection{Higher randomness}    The adjective  ``higher'' indicates that    algorithmic tools are replaced by tools from effective descriptive theory.  See e.g.\ \cite[Ch.\ 9]{Nies:book} for background. The work of the Madison group described in Section~\ref{s:MG convergence} can be adapted  to this setting. For a higher version of density randomness,  instead of $\PPI$ classes we now look at $\VI$ classes containing the real in question. Similar to the foregoing case of $\Pi^0_n$ classes,   it does not matter  whether the $\VI$ class  is  closed.

We use the following fact due  to  Greenberg (personal communication). It is a higher analog of the original weaker version of  Prop.\ \ref{prop:PC random upper density}(i) proved in Bienvenu et al.\ \cite[Prop.\ 5.4]{Bienvenu.Greenberg.ea:nd}. The hypothesis on $Z$  could be weakened  to a higher notion  of partial computable randomness as well.

\begin{proposition}[Greenberg, 2013] \label{prop:higher ML random upper density} Let $\+ C \sub \cantor$ be $\VI$. Let  $Z \in \+ C$ be $\QI$-ML-random. Then $\ol \varrho_2(\+ C \mid Z) =1$. \end{proposition} 
\begin{proof} If $\ol \varrho_2(\+ C \mid Z) < 1$ then there is    a positive rational  $q<1$ and $n^*$ such that for   all $n\ge n^*$ we have $\leb_{Z\upr n}(\+ C)< q$.  Choose a rational $r$ with $q< r<1$. We define $\QI$-antichains in $U_n \sub \strcantor$, uniformly in $n$. Let $U_0 = \{\la Z \uhr {n^*} \ra\}$. Suppose $U_n$ has been defined.  For each $\sss \in U_n$, at a stage $\alpha$ such that $\leb_\sss(\+ C_\alpha) < q$, we obtain effectively  a hyper-arithmetical  antichain $V$   of extensions of $\sss$ such that $\+ C _\alpha \cap [\sss] \sub \Opcl V$ and $\leb_\sss(\Opcl V) < r$. Put $V$ into $U_{n+1}$. 

Clearly $\leb  \Opcl {U_n} \le r^n$ for each $n$. Also, $Z \in \bigcap_n \Opcl{U_n}$, so $Z$ is not $\QI$-ML-random.  \end{proof} 

A martingale $L\colon \strcantor \to \RR^+_0$ is called left-$\QI$ if $L(\sss)$ is a left-$\QI$ real uniformly in $\sss$.  We provide a higher analog of Theorem~\ref{thm:Madison}.
\begin{thm} Let $Z$ be  $\QI$-ML-random. The following are equivalent.

\bi \item[(i)] $\ul \varrho_2(\+ C \mid Z)=1 $ for each $\Sigma^1_1$ class $\+ C$ containing $Z$.
\item[(ii)] $\ul \varrho_2(\+ C \mid Z)=1 $ for each \emph{closed} $\Sigma^1_1$ class $\+ C$ containing $Z$.
\item[(iii)] Each left-$\QI$ martingale converges along  $Z$ to a finite value. \ei
\end{thm}

\begin{proof}  (iii) $\to$ (i): The measure of a $\VI$ set is left-$\VI$ in a uniform way (see e.g.\ \cite[Thm.\ 9.1.10]{Nies:book}). Therefore $M(\sss)=  1- \leb_\sss(\+ C)$ is a left-$\QI$ martingale. Since $M$ converges along $Z$, and since by Prop.\  \ref{prop:higher ML random upper density} $\liminf_n M(Z\uhr n) = 0$,  it converges along $Z$ to $0$.  This shows that $\ul \varrho_2(\+ C \mid Z)=1 $.

\n (ii) $\to$ (iii).   We follow the proof of  the Madison group's    Theorem~\ref{thm:Madison}  given  above. All stages $s$ are now interpreted as computable ordinals.  Computable functions are now  functions   $\OMC \to L_{\OMC}$ with $\Sigma_1$ graph. Constructions are now   assignments of recursive ordinals to instructions.

\begin{definition}  A \emph{$\QI$-Madison test} is a  $\Sigma_1$ over $L_{\omega_1^{CK}}$ function  $\seq {U_s}_{ s< \omega_1^{CK}}$ mapping ordinals to (hyperarithmetical)  subsets of $\strcantor$  such that $U_0 = \ES$,    for each  stage $s$ we have $\weight (U_s)  \le c$ for some constant $c$, and for  all strings $\sss, \tau$, 
\bi \item[(a)]  $\tau \in U_s - U_{s+1} \to \ex \sss \prec \tau \, [ \sss \in U_{s+1} - U_s]$
\item [(b)] $\weight (\sss^\prec \cap U_s) > \tp{-\sssl} \to \sss \in U_s$. 
\ei 
Also $U_t(\sigma) = \lim_{s < t} U_s(\sss)$ for each limit ordinal $t$. 
\end{definition}

The following well-known fact can be proved similar to \cite[1.9.19]{Nies:book}.
\begin{lemma} \label{lem:extend open} Let $\+ A \sub \cantor$ be a hyperarithmetical open set. Given  a rational   $q$ with $q >  \leb A$, we can effectively determine from $\+ A, q$  a hyperarithmetical open $\+ S \supseteq \+ A$ with $\leb \+ S = q$. \end{lemma}
  We provide an analog of Lemma~\ref{lem: density to Madison}.  Its proof is  a variant of the former argument.
\begin{lemma} \label{lem: density to higher Madison} Let $Z$ be a $\QI$ ML-random such that $\ul \varrho_2(\+ C \mid Z)=1 $ for each \emph{closed} $\Sigma^1_1$ class $\+ C$ containing $Z$. Then $Z$ passes each $\QI$-Madison test. \end{lemma}

The sets $\+ A^k_{\sss,s}$  are now hyperarithmetical open sets computed from $k,\sss, s$.  Suppose $\sss \in U_{s+1} - U_s$. The set $\wt {\+ A}^k_{\sss,s}$  is defined as before. To effectively obtain $ \+ A^k_{\sss , s+1}$, we apply Lemma~\ref{lem:extend open} to  add mass from $[\sss]$ to   $\wt  {\+ A}^k_{\sss , s+1}$ in order to ensure that $\leb (  {\+ A}^k_{\sss,s+1}) = \tp{- \sssl -k}$. 

As before, let   $\+ S^k_t = \bigcup_{\sss \in U_t} \+ A^k_{\sss,t}$. Then $\+ S^k_t \sub \+ S^k_{t+1}$ by condition (a) on $\QI$-Madison tests. Clearly $\leb \+ S^k_t \le  \tp {-k} \weight (U_t)  \le \tp{-k}$. So $\+ S^k = \bigcup_{t< \OMC} \+ S ^k_t$ determines  a $\QI$ ML-test. 

By construction 
$\ul \varrho_2(\cantor - \+ S^k \mid Z) \le 1- \tp{-k}$. Since $Z$ is ML-random we have $Z \not \in \+ S^k$ for some $k$. So $\ul \varrho_2(\+ C \mid Z) < 1 $ for the   {closed} $\Sigma^1_1$ class $\+ C = \cantor - \+ S_k$ containing $Z$.

The analog of Lemma~\ref{Madison to  MG convergence} 
 also holds.  
\begin{lemma} Suppose that $Z$ passes each $\QI$-Madison test. Then every left-$\QI$  martingale $L$ converges along $Z$. \end{lemma}
We wrote  the proof  of Lemma~\ref{Madison to  MG convergence}  in such a way  that this works.  If  $L\colon \strcantor \to \RR$ is a  left-$\QI$ martingale, then $L(\sss) = \sup_s L_s(\sss)$ for a   non-decreasing sequence $\seq{L_s}$ of hyperarithmetical martingales computed uniformly from $s< \OMC$.  The    labelling functions $\gamma_s \colon \, U_s \to\OMC$ are now  uniformly hyperarithmetical. 

We may assume that $L_t(\sss) = \lim_{s<t} L_s(\sss)$ for each limit ordinal $t$. This implies  $U_t(\sigma) = \lim_{s < t} U_s(\sss)$ for each limit ordinal $t$ as required in the definition of higher Madison tests.  
%
\end{proof}

%

\def\cprime{$'$}


\end{document}